\documentclass[a4,12pt]{amsart}
\usepackage{amssymb,amstext,amscd,amsfonts,mathdots,mathrsfs}
\numberwithin{equation}{section}
\usepackage{hyperref}
\usepackage{setspace}
\usepackage[all]{xy}
\usepackage{amsthm}
\usepackage{enumerate}
\usepackage{diagbox}
\usepackage{indentfirst}
\usepackage{color}
\usepackage{colortbl}
\usepackage{amsmath}
\usepackage{tikz}
\usetikzlibrary{shapes.geometric, arrows}
\usetikzlibrary{calc}

\newtheorem{theorem}{Theorem}[section]
\newtheorem{theorem*}{Theorem}
\newtheorem{corollary}[theorem]{Corollary}
\newtheorem{corollary*}[theorem*]{Corollary}
\newtheorem{lemma}[theorem]{Lemma}
\newtheorem{proposition}[theorem]{Proposition}

\theoremstyle{definition}
\newtheorem{definition}[theorem]{Definition}
\newtheorem{remark}[theorem]{Remark}

\newtheorem*{question*}{Question}

\newtheorem*{conjecture*}{Conjecture}
\newtheorem{example}[theorem]{Example}

\newtheorem*{notation*}{Notation}

\newtheorem*{claim*}{Claim}

\begin{document}
\setlength{\baselineskip}{16pt}

\title{$\tau$-tilting finiteness of two-point algebras II}
\author{Qi Wang}
\address{Yau Mathematical Sciences Center, Tsinghua University, Beijing 100084, China.}
\email{infinite-wang@outlook.com (cc:infinite-wang@tsinghua.edu.cn)}


\begin{abstract}
In this paper, we explain a strategy on $g$-vectors to discover some new minimal $\tau$-tilting infinite two-point algebras. Consequently, the $\tau$-tilting finiteness of various two-point monomial algebras, including all radical cube zero cases, could be determined. Moreover, we find that the derived equivalence class of the Kronecker algebra contains only itself and its opposite algebra.
\end{abstract}
\maketitle

\section{Introduction}
$\tau$-tilting theory was introduced by Adachi, Iyama and Reiten \cite{AIR} in 2014, in which the authors constructed support $\tau$-tilting modules as a generalization of classical tilting modules. Here, $\tau$ stands for the Auslander-Reiten translation. After years of research, $\tau$-tilting theory is now considered as one of the main tools in the representation theory of finite-dimensional algebras. Support $\tau$-tilting modules, for example, are in bijection with several other objects in representation theory, including two-term silting complexes, functorially finite torsion classes and left finite semibricks. We refer to \cite{Asai}, \cite{AI-silting}, \cite{BST-max-green-seq}, \cite{DIRRT} and \cite{EJR-reduction} for more materials.

A natural question in $\tau$-tilting theory was proposed by Demonet, Iyama and Jasso \cite{DIJ-tau-tilting-finite}, that is, when does a finite-dimensional algebra admit only finitely many support $\tau$-tilting modules. In this context, we could define $\tau$-tilting finite/infinite algebras. It is known that a $\tau$-tilting finite algebra $\Lambda$ admits a nice behavior, for example, the number of bricks over $\Lambda$ is finite \cite{DIRRT}, the length of bricks over $\Lambda$ is bounded \cite{Schroll-Treffinger}, all torsion classes over $\Lambda$ are functorially finite \cite{DIJ-tau-tilting-finite}, all semibricks over $\Lambda$ are left finite \cite{Asai}, etc. This makes it attractive to investigate the $\tau$-tilting finiteness of algebras, which has already been done for radical square zero algebras \cite{Ada-rad-square-0}, Brauer graph algebras \cite{AAC-Brauer-graph-alg}, preprojective algebras of Dynkin type \cite{Mizuno-preprojective-alg}, cycle finite algebras \cite{MS-cycle-finite}, biserial algebras \cite{Mousavand-biserial-alg}, gentle algebras \cite{P-gentle}, minimal wild two-point algebras \cite{W-two-point}, Schur algebras \cite{W-schur}, simply connected algebras \cite{W-simply}, cluster-tilted algebras \cite{Z-tilted}, and so on.

In this paper, we continue our research on the $\tau$-tilting finiteness of two-point algebras, i.e., algebras with exactly two simple modules. We focus on the finiteness of two-term silting complexes instead of support $\tau$-tilting modules because of the bijection mentioned above. The main strategy is briefly described here, and the details are provided in Subsection 3.1. Let $\mathcal{K}_{\Lambda}:=\mathsf{K^b(proj}\ \Lambda)$ be the homotopy category of bounded complexes of finitely generated projective modules over a two-point algebra $\Lambda$. We consider a left mutation chain $(T_i)_{i\in \mathbb{N}}$ of two-term silting complexes in $\mathcal{K}_{\Lambda}$ and check the corresponding endomorphism algebras $\mathsf{End}_{\mathcal{K}_{\Lambda}}\ (T_i)$. With some additional assumptions, we find that
\begin{center}
$\vcenter{\xymatrix@C=1.2cm@R=1.2cm{T_1\ar[r]\ar@{.>}[d]|-{\mathsf{End}}&T_2\ar[r]\ar@{.>}[d]|-{\mathsf{End}}&
T_3\ar[r]\ar@{.>}[d]|-{\mathsf{End}}&\cdots \ar[r]&T_{2k-1}\ar[r]\ar@{.>}[d]|-{\mathsf{End}}&T_{2k}\ar[r]\ar@{.>}[d]|-{\mathsf{End}}&\cdots\\
\Gamma&\Gamma^{\text{op}}&\Gamma&\cdots&\Gamma&\Gamma^{\text{op}}&\cdots}}$
\end{center}
where $\xymatrix@C=1.2cm{T\ar@{.>}[r]|-{\mathsf{End}}&\Gamma}$ indicates $\mathsf{End}_{\mathcal{K}_{\Lambda}}\ (T)\simeq \Gamma$. Such a phenomenon that $\Gamma$ and $\Gamma^{\text{op}}$ appear alternately could make the chain $(T_i)_{i\in \mathbb{N}}$ to be an infinite chain, or equivalently, make $\Lambda$ to be $\tau$-tilting infinite. For example, the Kronecker algebra, which is well-known for being $\tau$-tilting infinite, admits exactly such a phenomenon. This case has been verified in Example \ref{example-kronecker-alg}.

Since monomial algebras are often used as a test class for a new phenomenon, we may apply the above strategy to two-point monomial algebras as an attempt toward all two-point algebras. We consider the following crucial examples of quivers.
\begin{center}
$Q_1: \xymatrix@C=0.8cm{\bullet \ar[r]^{\mu}& \bullet\ar@(ur,dr)^{\beta}}$, \ \ \ $Q_2: \xymatrix@C=0.8cm{\bullet \ar@(dl,ul)^{\alpha}\ar[r]^{\mu}& \bullet\ar@(ur,dr)^{\beta}}$, \ \ \ $Q_3: \xymatrix@C=0.8cm{\bullet\ar[r]^{\mu}& \bullet \ar@(ul,ur)^{\beta_1}\ar@(dl,dr)_{\beta_2}}$.
\end{center}
Let $K$ be an algebraically closed field. We define some two-point monomial algebras:
\begin{itemize}
\item $\Omega_1:=KQ_1/\langle\beta^4\rangle$;
\item $\Omega_2:=KQ_2/\langle\alpha^2, \beta^2\rangle$;
\item $\Omega_3:=KQ_3/\langle\beta_1^2, \beta_2^2, \beta_1\beta_2, \beta_2\beta_1\rangle$.
\end{itemize}

Recall that an algebra $\Lambda$ is called minimal $\tau$-tilting infinite if $\Lambda$ is $\tau$-tilting infinite, but any proper quotient of $\Lambda$ is $\tau$-tilting finite. By applying the above strategy, we have
\begin{theorem}
The two-point algebra $\Omega_i$ is minimal $\tau$-tilting infinite for $i=1,2,3$.
\end{theorem}
\begin{proof}
See Example \ref{Omega-2}, Proposition \ref{Omega-3} and Proposition \ref{Omega-1}.
\end{proof}

The first main result of this paper is stated as follows.
\begin{theorem}[Theorem \ref{result-rad-cube-zero}]
Let $\Lambda$ be an arbitrary two-point connected monomial algebra with radical cube zero. Then, $\Lambda$ is $\tau$-tilting finite if and only if it does not have one of the Kronecker algebra, $\Omega_3$ and $\Omega_3^{\emph{op}}$ as a quotient algebra.
\end{theorem}

We then aim to release the condition $\mathsf{rad}^3=0$ in Theorem 1.2. We obtain the following characterization for $\tau$-tilting finite two-point monomial algebras associated with $Q_1, Q_2$.
\begin{theorem}[Theorem \ref{result-Q1} and Theorem \ref{result-Q2}]
Let $\Lambda=KQ_i/I$ be a monomial algebra.
\begin{enumerate}
  \item[(1)] If $i=1$, then $\Lambda$ is $\tau$-tilting finite if and only if it does not have $\Omega_1$ as a quotient.
  \item[(2)] If $i=2$, then $\Lambda$ is $\tau$-tilting finite if and only if it does not have one of $\Omega_1$, $\Omega_1^{\emph{op}}$ and $\Omega_2$ as a quotient.
\end{enumerate}
\end{theorem}

One may use our approach to find more minimal $\tau$-tilting infinite two-point monomial algebras $\Omega$ satisfying $\mathsf{rad}^3\ \Omega\neq 0$ (for example, $\Omega_4$, $\Omega_5$ in Section 4) so that the restriction $\mathsf{rad}^3=0$ could be generalized in some cases. We mention that there is nothing new, but only more complicated analysis and calculations on left mutations and endomorphism algebras. In particular, we get a criterion for the $\tau$-tilting finiteness of two-point monomial algebras with $\mathsf{rad}^5=0$ associated with the following quiver:
\begin{center}
$\xymatrix@C=0.8cm{\bullet \ar@<0.5ex>[r]^{\mu}\ar@(dl,ul)^{\alpha}&\bullet \ar@<0.5ex>[l]^{\nu}\ar@(ur,dr)^{\beta}}$,
\end{center}
see Theorem \ref{result-Q(1111)} for details.

In the process of showing our main results, we find the following fact. This may be independent of our original motivation, but it seems to be interesting in the representation theory of silting-connected algebras.
\begin{theorem}[Proposition \ref{result-derived-eq-kronecker}]
Let $\Delta_1:=K(\xymatrix@C=0.7cm{1\ar@<0.5ex>[r]^{ }\ar@<-0.5ex>[r]_{ }&2})$ be the Kronecker algebra. Then, the derived equivalence class of $\Delta_1$ contains only $\Delta_1$ and $\Delta_1^{\emph{op}}$.
\end{theorem}

\section{Preliminaries}\label{section-2}
Let $\Lambda=KQ/I$ be a bound quiver algebra with a finite connected quiver $Q$ and an admissible ideal $I$ over an algebraically closed field $K$. We denote by $\mathsf{rad}\ \Lambda$ the Jacobson radical of $\Lambda$ and by $\Lambda^{\text{op}}$ the opposite algebra of $\Lambda$. We denote by $\mathsf{mod}\ \Lambda$ the category of finitely generated right $\Lambda$-modules and by $\mathsf{proj}\ \Lambda$ the full subcategory of $\mathsf{mod}\ \Lambda$ consisting of projective $\Lambda$-modules. A relation $\rho=\sum_{i=1}^{m}\lambda_i\omega_i$ in $I$ is a $K$-linear combination of paths $\omega_i$ of length at least two having the same source and target, where the $\lambda_i$ are scalars and not all zero. If $m=1$, then $\rho$ is called a monomial relation. We call $\Lambda$ a monomial algebra if the admissible ideal $I$ is generated by a set of monomial relations. For more background on the quiver representation theory of finite-dimensional algebras, see \cite{ASS}.

\subsection{Silting theory}
Let $\mathcal{K}_{\Lambda}=\mathsf{K^b(proj}\ \Lambda)$ be the homotopy category of bounded complexes of finitely generated projective $\Lambda$-modules. For any $T\in \mathcal{K}_{\Lambda}$, we denote by $\mathsf{thick}\ T$ the smallest thick subcategory of $\mathcal{K}_{\Lambda}$ containing $T$. Let $\mathsf{add}(T)$ be the full subcategory of $\mathcal{K}_{\Lambda}$ whose objects are direct summands of finite direct sums of copies of $T$.

\begin{definition}[{\cite[Definition 2.1]{AI-silting}}]
Let $T\in \mathcal{K}_{\Lambda}$. Then,
\begin{enumerate}
  \item $T$ is called presilting if $\mathsf{Hom}_{\mathcal{K}_{\Lambda}}(T,T[i])=0$ for any $i>0$.
  \item $T$ is called silting if $T$ is presilting and $\mathsf{thick}\ T=\mathcal{K}_{\Lambda}$.
  \item $T$ is called tilting if $T$ is silting and $\mathsf{Hom}_{\mathcal{K}_{\Lambda}}(T,T[i])=0$ for any $i<0$.
\end{enumerate}
\end{definition}

We recall the definition of irreducible left silting mutations of silting complexes. Let $T=T_1\oplus \cdots \oplus T_j\oplus\cdots\oplus T_n$ be a basic silting complex in $\mathcal{K}_{\Lambda}$ with an indecomposable direct summand $T_j$. We take a minimal left $\mathsf{add}(T/T_j)$-approximation $\pi$ and a triangle
\begin{center}
$T_j\overset{\pi}{\longrightarrow}Z\longrightarrow \mathsf{cone}(\pi)\longrightarrow T_j[1]$,
\end{center}
where $\mathsf{cone}(\pi)$ is the mapping cone of $\pi$. Then, it is known from \cite[Theorem 2.31]{AI-silting} that $\mathsf{cone}(\pi)$ is indecomposable and $\mu_j^-(T):=\mathsf{cone}(\pi)\oplus (T/T_j)$ is again a basic silting complex in $\mathcal{K}_{\Lambda}$. We call $\mu_j^-(T)$ the irreducible left silting mutation of $T$ with respect to $T_j$, or simply, the left mutation of $T$ with respect to $T_j$. Dually, we define the irreducible right silting mutation $\mu_j^+(T)$ of $T$ with respect to $T_j$.

We denote by $\mathsf{silt}\ \Lambda$ the set of isomorphism classes of basic silting complexes in $\mathcal{K}_{\Lambda}$. For any $T,S \in \mathsf{silt}\ \Lambda$, we say $T\geqslant S$ if $\mathsf{Hom}_{\mathcal{K}_{\Lambda}}(T,S[i])=0$ for any $i>0$. This actually gives a partial order on the set $\mathsf{silt}\ \Lambda$. Moreover, it is shown in \cite[Theorem 2.35]{AI-silting} that $S$ is a left mutation of $T$ if and only if $T$ is a right mutation of $S$, if and only if, $T>S$ and there is no $U\in \mathsf{silt}\ \Lambda$ such that $T>U>S$.

Next, we restrict our attention to two-term silting complexes. A complex in $\mathcal{K}_{\Lambda}$ is called \emph{two-term} if it is homotopy equivalent to a complex $T$, which is concentrated in degrees $0$ and $-1$, i.e.,
\begin{center}
$T=( T^{-1}\overset{d_T^{-1}}{\longrightarrow } T^0 ):=\xymatrix@C=0.7cm{\cdots\ar[r]&0\ar[r]&T^{-1}\ar[r]^{d_T^{-1}}&T^0\ar[r]&0\ar[r]&\cdots}$.
\end{center}
Let $\mathsf{2\text{-}silt}\ \Lambda$ be the subset of two-term complexes in $\mathsf{silt}\ \Lambda$. Obviously, $\mathsf{2\text{-}silt}\ \Lambda$ is a poset under the partial order $\geqslant$ on $\mathsf{silt}\ \Lambda$. We then denote by $\mathcal{H}(\mathsf{2\text{-}silt}\ \Lambda)$ the Hasse quiver of $\mathsf{2\text{-}silt}\ \Lambda$, which is compatible with the left/right mutation of two-term silting complexes.

\begin{proposition}[{\cite[Lemma 2.25, Theorem 2.27]{AI-silting}}]\label{prop-empty}
Let $T=(T^{-1}\rightarrow T^0)\in \mathsf{2\text{-}silt}\ \Lambda$. Then, we have $\mathsf{add}\ \Lambda = \mathsf{add}\ (T^0 \oplus T^{-1})$ and $\mathsf{add}\  T^0\cap \mathsf{add}\  T^{-1}=0$.
\end{proposition}

Let $\Lambda$ be a bound quiver algebra with $n$ simple modules and $P_1, P_2,\ldots, P_n$ the pairwise non-isomorphic indecomposable projective $\Lambda$-modules. We denote by $[P_1]$, $[P_2]$, $\ldots$, $[P_n]$ the isomorphism classes of indecomposable complexes concentrated in degree 0. Then, the classes $[P_1], [P_2], \ldots, [P_n]$ in $\mathcal{K}_{\Lambda}$ induce a standard basis of the Grothendieck group $K_0(\mathcal{K}_{\Lambda})$. If a two-term complex $T$ in $\mathcal{K}_{\Lambda}$ is written as
\begin{center}
$\displaystyle\left ( \bigoplus_{i=1}^n P_i^{\oplus b_i}\longrightarrow \bigoplus_{i=1}^n P_i^{\oplus a_i} \right )$,
\end{center}
then the class $[T]$ can be identified by an integer vector
\begin{center}
$g(T) = (a_1-b_1, a_2-b_2, \ldots, a_n-b_n)\in \mathbb{Z}^n$,
\end{center}
which is called the $g$-vector of $T$. We always display $g$-vectors of two-term silting complexes as the direct sum of $g$-vectors of indecomposable two-term presilting complexes. The following is a critical statement.
\begin{proposition}[{\cite[Theorem 5.5]{AIR}}]\label{prop-g-vector-injection}
Let $T$ be a two-term silting complex in $\mathsf{2\text{-}silt}\ \Lambda$. Then, the map $T \mapsto g(T)$ is an injection.
\end{proposition}

\subsection{Connection with $\tau$-tilting theory}
We briefly review the fundamental definitions in $\tau$-tilting theory. Then, the connection between $\tau$-tilting theory and silting theory is explained, which will make the reason clear why we focus on two-term silting complexes.

\begin{definition}[{\cite[Definition 0.1]{AIR}}]
Let $M\in \mathsf{mod}\ \Lambda$ and $|M|$ be the number of isomorphism classes of indecomposable direct summands of $M$.
\begin{enumerate}
  \item $M$ is called $\tau$-tilting if $\mathsf{Hom}_\Lambda(M,\tau M)=0$ and $\left | M \right |=\left | \Lambda \right |$.
  \item $M$ is called support $\tau$-tilting if $M$ is a $\tau$-tilting $\left ( \Lambda/\Lambda e \Lambda\right )$-module for an idempotent $e$ of $\Lambda$. In this case, put $P:=e\Lambda$. Then, $(M,P)$ is called a support $\tau$-tilting pair.
\end{enumerate}
\end{definition}

We denote by $\mathsf{s\tau\text{-}tilt}\ \Lambda$ the set of isomorphism classes of basic support $\tau$-tilting $\Lambda$-modules. Then, an algebra $\Lambda$ is called $\tau$-tilting finite if the set $\mathsf{s\tau\text{-}tilt}\ \Lambda$ is finite. Otherwise, $\Lambda$ is said to be $\tau$-tilting infinite.

\begin{theorem}[{\cite[Theorem 3.2]{AIR}}]
There exists a poset isomorphism between $\mathsf{s\tau\text{-}tilt}\ \Lambda$ and $2\mathsf{\text{-}silt}\ \Lambda$. More precisely, the bijection is given by
\begin{center}
$\xymatrix@C=0.8cm@R=0.2cm{M\ar@{|->}[r]&(P_1\oplus P\overset{(f,0)}{\longrightarrow} P_0)}$,
\end{center}
where $(M,P)$ is the support $\tau$-tilting pair corresponding to $M$ and $P_1\overset{f}{\longrightarrow }P_0\overset{}{\longrightarrow }M\longrightarrow 0$ is the minimal projective presentation of $M$.
\end{theorem}

Immediately, we know that checking the $\tau$-tilting finiteness of an algebra $\Lambda$ is identical to checking the finiteness of $\mathsf{2\text{-}silt}\ \Lambda$. This is equivalent to finding either a finite connected component or an infinite left mutation chain in $\mathcal{H}(\mathsf{2\text{-}silt}\ \Lambda)$, due to the following statement.
\begin{proposition}[{\cite[Corollary 2.38]{AIR}}]
If the Hasse quiver $\mathcal{H}(\mathsf{2\text{-}silt}\ \Lambda)$ contains a finite connected component $\mathcal{C}$, then $\mathcal{H}(\mathsf{2\text{-}silt}\ \Lambda)=\mathcal{C}$.
\end{proposition}

In order to find a finite connected component of $\mathcal{H}(\mathsf{2\text{-}silt}\ \Lambda)$, we usually calculate the left mutations starting with $\Lambda$ because $\Lambda$ is the maximal element in the poset $\mathsf{2\text{-}silt}\ \Lambda$. Although such a left mutation is always silting, it is not always two-term. Hence, it is necessary to make clear when such a left mutation is out of $\mathsf{2\text{-}silt}\ \Lambda$. We have the following essential statement. Here, we denote by $|T|$ the number of isomorphism classes of indecomposable direct summands of $T$.
\begin{proposition}[{\cite[Corollary 3.8]{AIR}}]\label{prop-exactly-two}
Let $T$ be a two-term presilting complex in $\mathcal{K}_{\Lambda}$ with $|T|=|\Lambda|-1$. Then, $T$ is a direct summand of exactly two basic two-term silting complexes in $\mathsf{2\text{-}silt}\ \Lambda$.
\end{proposition}

It is worth mentioning that two-point algebras have a nice behavior towards the component of $\mathcal{H}(\mathsf{2\text{-}silt}\ \Lambda)$. We assume that $\Lambda$ is a two-point algebra and $T:=T_1\oplus T_2\in \mathsf{2\text{-}silt}\ \Lambda$. By Proposition \ref{prop-exactly-two}, both $\mu_1^-(\mu_1^-(T))$ and $\mu_2^-(\mu_2^-(T))$ are out of $\mathsf{2\text{-}silt}\ \Lambda$. Hence, there are only two possible orders for taking left mutations starting with $\Lambda$, i.e., either $1,2,1,2,1,2, \cdots$ or $2,1,2,1,2,1,\cdots$. In other words, the Hasse quiver $\mathcal{H}(\mathsf{2\text{-}silt}\ \Lambda)$ for a two-point algebra $\Lambda$ must be of the form
\begin{equation}\label{the-form-of-Hasse-quiver}
\vcenter{\xymatrix@C=0.7cm@R=0.2cm{&\mu_1^-(\Lambda) \ar[r]&\mu_2^-(\mu_1^-(\Lambda))\ar[r]&\mu_1^-(\mu_2^-(\mu_1^-(\Lambda)))\ar[r] &\cdots \ar[r]&\circ \ar[dr]&\\
\Lambda \ar[ur]\ar[dr]&&&&&&\Lambda[1].\\
&\mu_2^-(\Lambda)\ar[r]&\mu_1^-(\mu_2^-(\Lambda))\ar[r]& \mu_2^-(\mu_1^-(\mu_2^-(\Lambda)))\ar[r]&\cdots \ar[r]&\circ \ar[ur]&}}
\end{equation}
In this context, we call $1,2,1,2,1,2, \cdots$ and $2,1,2,1,2,1,\cdots$ the mutation orders of $\Lambda$.

\begin{remark}\label{silt-is-tilt}
Let $\Lambda$ be a two-point algebra and $P_1, P_2$ its indecomposable projective modules. Based on Proposition \ref{prop-empty}, a two-term silting complex $T$ appearing in mutation order $1,2,1,2,1,2, \cdots$ must be of the form $(P_1^{\oplus k}\rightarrow P_2^{\oplus\ell})$ for some integers $k,\ell$. If moreover, $\mathsf{Hom}_{\Lambda}(P_2,P_1)=0$, then $T$ is a tilting complex. We omit the similar argument for the two-term silting complexes appearing in mutation order $2,1,2,1,2,1,\cdots$.
\end{remark}

\subsection{Reduction on $\tau$-tilting finiteness}
We need several methods to check the $\tau$-tilting finiteness of algebras. First of all, we can distinguish the minimal case among all $\tau$-tilting infinite algebras.
\begin{definition}[{\cite{W-two-point}}]
An algebra $\Lambda$ is called minimal $\tau$-tilting infinite if $\Lambda$ is $\tau$-tilting infinite, but any proper quotient algebra of $\Lambda$ is $\tau$-tilting finite.
\end{definition}

It is known (for example, \cite[Theorem 2.6]{Ada-rad-square-0}) that any path algebra $KQ$ with a finite quiver $Q$ whose underlying graph is one of the Euclidean diagrams of type $\widetilde{\mathbb{A}}_n$, $\widetilde{\mathbb{D}}_n (n\geqslant 4)$, $\widetilde{\mathbb{E}}_6$, $\widetilde{\mathbb{E}}_7$, $\widetilde{\mathbb{E}}_8$, is minimal $\tau$-tilting infinite. Mousavand also introduced this notion independently in his recent work \cite{Mousavand-biserial-alg}, where the aim is to give a complete classification of minimal representation-infinite algebras in terms of $\tau$-tilting finiteness.

Secondly, we recall the following reduction statements to emphasize the importance of minimal $\tau$-tilting infinite algebras.
\begin{lemma}[{\cite[Theorem 4.2]{DIJ-tau-tilting-finite}, \cite[Theorem 5.12]{DIRRT}}]\label{lem-quotient-and-idempotent}
If $\Lambda$ is $\tau$-tilting finite, then
\begin{enumerate}
  \item the quotient algebra $\Lambda/I$ is $\tau$-tilting finite for any two-sided ideal $I$ of $\Lambda$,
  \item the idempotent truncation $e\Lambda e$ is $\tau$-tilting finite for any idempotent $e$ of $\Lambda$.
\end{enumerate}
\end{lemma}

If we take a special two-sided ideal $I$ of $\Lambda$ in Lemma \ref{lem-quotient-and-idempotent} (1), the $\tau$-tilting finiteness of $\Lambda$ is reduced to that of $\Lambda/I$. This effective technical method is provided by Eisele, Janssens and Raedschelders \cite{EJR-reduction}.
\begin{lemma}[{\cite[Theorem 1]{EJR-reduction}}]\label{lem-center}
Let $I$ be a two-sided ideal generated by central elements included in $\mathsf{rad}\ \Lambda$. Then, the $g$-vectors of two-term silting complexes over $\Lambda$ coincide with the ones over $\Lambda/I$, as do the left/right silting mutations.
\end{lemma}

According to the $\Lambda$-duality $(-)^\ast:=\mathsf{Hom}_\Lambda(-,\Lambda)$, we have
\begin{lemma}[{\cite[Theorem 2.14]{AIR}}]\label{lem-oppsite-algebra}
There exists a poset anti-isomorphism between $\mathsf{2\text{-}silt}\ \Lambda$ and $\mathsf{2\text{-}silt}\ \Lambda^{\emph{op}}$.
\end{lemma}

Lastly, we mention another technical method investigated lately by Aihara and the author \cite{Aihara-Wang}. Let $\{e_1,e_2, \cdots, e_n\}$ be a complete list of pairwise orthogonal primitive idempotents of $\Lambda$. If there exists an algebra isomorphism $\sigma: \Lambda^{\text{op}}\rightarrow \Lambda$, then $\sigma$ induces a permutation on $\{1,2, \cdots, n\}$ by $\sigma(e_i^\ast)=e_j$. We set $S_\sigma:=[1]\circ \sigma \circ (-)^\ast$.
\begin{lemma}[{\cite[Theorem 1.2, Corollary 1.5]{Aihara-Wang}}]\label{lem-symmetry}
The functor $S_\sigma$ induces an anti-automorphism of the poset $\mathsf{2\text{-}silt}\ \Lambda$. For any indecomposable two-term presilting complex $T$ with $g(T)=(c_1,c_2, \cdots, c_n)$, we have $g(S_\sigma(T))=-(c_{\sigma(1)}, c_{\sigma(2)}, \cdots, c_{\sigma(n)})$.
\end{lemma}

\section{Main Results}
\subsection{Main strategy}
Let $\mathsf{D^b(mod}\ \Lambda)$ be the derived category of $\mathsf{mod}\ \Lambda$; it is actually the localization of $\mathcal{K}_{\Lambda}$ with respect to quasi-isomorphisms. We note that both $\mathsf{D^b(mod}\ \Lambda)$ and $\mathcal{K}_{\Lambda}$ are triangulated categories. Let us start with the following easy observation.

\begin{proposition}[{\cite[Lemma 2.8]{Au-silting}}]\label{lemma-original}
Let $\Lambda$ and $\Gamma$ be algebras with a triangle equivalence $F: \mathsf{D^b(mod}\ \Lambda)\longrightarrow \mathsf{D^b(mod}\ \Gamma)$. Then, the following statements hold.
\begin{enumerate}
  \item $F$ sends silting complexes in $\mathcal{K}_{\Lambda}$ to silting complexes in $\mathcal{K}_{\Gamma}$.
  \item $F$ preserves the partial order on the set of silting complexes.
  \item If $T$ is a silting complex in $\mathcal{K}_{\Lambda}$, then $F(\mu_i^-(T))\simeq \mu_i^-(F(T))$.
\end{enumerate}
\end{proposition}

It is known from Rickard \cite{Rickard-tilting-complex} that the above triangle equivalence $F: \mathsf{D^b(mod}\ \Lambda)\rightarrow \mathsf{D^b(mod}\ \Gamma)$ exists if and only if there is a tilting complex $T\in \mathcal{K}_{\Lambda}$ such that $\Gamma\simeq\mathsf{End}_{\mathcal{K}_{\Lambda}}\ (T)$. If this is the case, we say that $\Lambda$ is derived equivalent to $\Gamma$. In this paper, we restrict our interests to two-point algebras.

Let $\Lambda$ be a two-point algebra and $T:=T_1\oplus T_2$ a tilting complex in $\mathcal{K}_{\Lambda}$. We define $\Gamma:=\mathsf{End}_{\mathcal{K}_{\Lambda}}\ (T)$, and denote by $P_1$, $P_2$ the indecomposable projective $\Gamma$-modules. Then, the triangle equivalence
\begin{center}
$F : \mathsf{D^b(mod}\ \Lambda) \overset{\sim}{\longrightarrow} \mathsf{D^b(mod}\ \Gamma)$
\end{center}
is given by mapping $T\mapsto \Gamma$. In this way, the indecomposable direct summand $T_i$ of $T$ is mapped to the indecomposable projective module $P_i$ of $\Gamma$. It turns out that $F$ naturally induces an isomorphism $K_0(\mathcal{K}_{\Lambda})\longrightarrow K_0(\mathcal{K}_\Gamma)$ of Grothendieck groups by $[T_i] \mapsto [P_i]$ for $i=1,2$. The following left mutation chains may be considered:
\begin{center}
$\vcenter{\xymatrix@C=0.4cm@R=1cm{T_1\oplus T_2\ar[rr]\ar[d]^{F}&&S\oplus T_2\ar[d]^{F}&\in \mathcal{K}_\Lambda\\
P_1\oplus P_2 \ar[rr]&&Q\oplus P_2& \in \mathcal{K}_\Gamma}}$.
\end{center}
Notice that $S\oplus T_2$ and $Q\oplus P_2$ are again silting but not necessarily tilting. Then, the following statements are implicit in Proposition \ref{lemma-original}.

\begin{lemma}\label{lemma-endomorphism-algebra}
Under the above setting, we have
\begin{center}
$\mathsf{End}_{\mathcal{K}_{\Lambda}}\ (S\oplus T_2)\simeq \mathsf{End}_{\mathcal{K}_\Gamma}\ (Q\oplus P_2)$.
\end{center}
\end{lemma}
\begin{proof}
Since $F$ is a triangle equivalence, it preserves both triangles and minimal left approximations. Then, the triangle
\begin{center}
$T_1 \longrightarrow T_2^{\oplus k}\longrightarrow S\longrightarrow T_1[1] \in \mathcal{K}_{\Lambda}$
\end{center}
is mapped under $F$ to the triangle
\begin{center}
$P_1 \longrightarrow P_2^{\oplus k}\longrightarrow Q\longrightarrow P_1[1] \in \mathcal{K}_\Gamma$.
\end{center}
Then, it is not difficult to find the statement.
\end{proof}

\begin{lemma}\label{lemma-g-vector}
Assume that the above $T_i, S, P_i, Q$ are two-term complexes. Then, we have $g(S)=kg(T_2)-g(T_1)$ if $g(Q)=kg(P_2)-g(P_1)$.
\end{lemma}
\begin{proof}
By the definition of $g$-vectors, it is obvious that $g(S)$ must be of the form $kg(T_2)-g(T_1)$. Then, it is enough to notice that there is an isomorphism $K_0(\mathcal{K}_{\Lambda})\to K_0(\mathcal{K}_\Gamma)$ of Grothendieck groups by $[T_i] \mapsto [P_i]$ for $i=1,2$.
\end{proof}

We provide two examples to illustrate how the above lemmas can be applied.

\begin{example}\label{example-kronecker-alg}
Let $\Delta_1=K(\xymatrix@C=0.7cm{1\ar@<0.5ex>[r]^{\mu_1}\ar@<-0.5ex>[r]_{\mu_2}&2})$ be the Kronecker algebra. It is known as a minimal $\tau$-tilting infinite algebra. We look at the $g$-vectors of two-term silting complexes appearing in mutation order $2,1,2,1,2, \cdots$. Let $P_1$ and $P_2$ be the indecomposable projective $\Delta_1$-modules. We have
\begin{center}
$P_1=$\scalebox{0.8}{$\vcenter{\xymatrix@C=0.01cm@R=0.2cm{
&e_1\ar@{-}[dl]\ar@{-}[dr]&\\
\mu_1&&\mu_2}}
\simeq \vcenter{\xymatrix@C=0.02cm@R=0.2cm{
&1\ar@{-}[dl]\ar@{-}[dr]&\\
2&&2}}$} and $P_2=e_2\simeq 2$.
\end{center}
In the mutation order $2,1,2,1,2, \cdots$, it follows from Remark \ref{silt-is-tilt} that any two-term silting complex is tilting since $\mathsf{Hom}_{\Delta_1}(P_1,P_2)=0$. It is easy to check that
\begin{center}
$\mu_{2}^-(\Delta_1)=\left [\begin{smallmatrix}
0\longrightarrow P_1\\
\oplus \\
P_2\overset{f}{\longrightarrow} P_1^{\oplus 2}
\end{smallmatrix}  \right ]$ with $f=\left ( \begin{smallmatrix}
\mu_1\\
\mu_2
\end{smallmatrix} \right )$.
\end{center}
Set $X:=(0\longrightarrow P_1)$ and $Y:=(P_2\overset{f}{\longrightarrow} P_1^{\oplus 2})$. Then,
\begin{center}
$\mathsf{Hom}_{\mathcal{K}_{\Delta_1}}(Y,X)=0$, $\mathsf{Hom}_{\mathcal{K}_{\Delta_1}}(X,X)=\left \{\left (0, 1\right )\right \}$,
\end{center}

\begin{center}
$\mathsf{Hom}_{\mathcal{K}_{\Delta_1}}(Y,Y)=\left \{\left (1, \left ( \begin{smallmatrix}
1&0\\
0&1
\end{smallmatrix} \right )\right )\right \}$,
\end{center}

\begin{center}
$\mathsf{Hom}_{\mathcal{K}_{\Delta_1}}(X,Y)=\mathsf{span} \left \{\left (0, \left ( \begin{smallmatrix}
1\\
0
\end{smallmatrix} \right )\right ), \left (0, \left ( \begin{smallmatrix}
0\\
1
\end{smallmatrix} \right )\right )\right \}$.
\end{center}
By direct calculation, we find that $\mathsf{End}_{\mathcal{K}_{\Delta_1}}(\mu_{2}^-(\Delta_1))\simeq \Delta_1^{\text{op}}=K(\xymatrix@C=0.7cm{1&2\ar@<0.5ex>[l]^{\ }\ar@<-0.5ex>[l]_{\ }})$. We then use Lemma \ref{lemma-endomorphism-algebra} to find that
\begin{center}
$\mathsf{End}_{\mathcal{K}_{\Delta_1}}(\mu_1^-(\mu_{2}^-(\Delta_1)))\simeq \mathsf{End}_{\mathcal{K}_{\Delta_1^{\text{op}}}}(\mu_{1}^-(\Delta_1^{\text{op}}))$.
\end{center}
Since the calculation of $\mu_{1}^-(\Delta_1^{\text{op}})$ is exactly the same as that of $\mu_{2}^-(\Delta_1)$, we deduce that
\begin{center}
$g(\mu_1^-(\mu_{2}^-(\Delta_1)))=\begin{smallmatrix}
(3,-2)\\
\oplus \\
(2,-1)
\end{smallmatrix}$
\end{center}
by Lemma \ref{lemma-g-vector}. It is obvious that $\mathsf{End}_{\mathcal{K}_{\Delta_1^{\text{op}}}}(\mu_{1}^-(\Delta_1^{\text{op}}))\simeq (\Delta_1^{\text{op}})^{\text{op}}\simeq \Delta_1$ and hence,
\begin{center}
$\begin{aligned}\mathsf{End}_{\mathcal{K}_{\Delta_1}}(\mu_2^-(\mu_1^-(\mu_{2}^-(\Delta_1))))
&\simeq \mathsf{End}_{\mathcal{K}_{\Delta_1^{\text{op}}}}(\mu_2^-(\mu_{1}^-(\Delta_1^{\text{op}})))\\
&\simeq \mathsf{End}_{\mathcal{K}_{\Delta_1}}(\mu_2^-(\Delta_1))\\
&\simeq \Delta_1^{\text{op}}
\end{aligned}$.
\end{center}

By repeating the above procedures, we obtain a left mutation chain
\begin{center}
$\Delta_1=:T_0\longrightarrow T_1\longrightarrow T_2\longrightarrow T_3\longrightarrow \cdots  \in \mathcal{H}(\mathsf{2\text{-}silt}\ \Delta_1)$,
\end{center}
where $T_i$ is always tilting, $\mathsf{End}_{\mathcal{K}_{\Delta_1}}(T_{2n})\simeq \Delta_1$ and $\mathsf{End}_{\mathcal{K}_{\Delta_1}}(T_{2n+1})\simeq \Delta_1^{\text{op}}$. Since a basic two-term silting complex $T$ is uniquely determined by its $g$-vector $g(T)$ (see Proposition \ref{prop-g-vector-injection}), we consider the $g$-vector of $T_i$. Then, we have the following chain
\begin{center}
$\begin{smallmatrix}
(1,0)\\
\oplus \\
(0,1)
\end{smallmatrix}\rightarrow \begin{smallmatrix}
(1,0)\\
\oplus \\
(2,-1)
\end{smallmatrix}\rightarrow \begin{smallmatrix}
(3,-2)\\
\oplus \\
(2,-1)
\end{smallmatrix} \rightarrow \begin{smallmatrix}
(3,-2)\\
\oplus \\
(4,-3)
\end{smallmatrix} \rightarrow \cdots\rightarrow g(T_{2n})\rightarrow g(T_{2n+1})\rightarrow \cdots$,
\end{center}
where the $g$-vector $g(T_{2n})$ is given by $\left ( \begin{smallmatrix}
3& -2 \\
2& -1
\end{smallmatrix} \right )^n$ so that
\begin{center}
$g(T_{2n})=\begin{smallmatrix}
(2n+1, -2n )\\
\oplus \\
(2n, -2n+1)
\end{smallmatrix}$ and $g(T_{2n+1})=\begin{smallmatrix}
(2n+1, -2n )\\
\oplus \\
(2n+2, -2n-1)
\end{smallmatrix}$.
\end{center}
It is obvious that the above chain of $g$-vectors cannot reach $(-1,0)\oplus (0,-1)$. This implies that $\Delta_1$ is $\tau$-tilting infinite.
\end{example}

\begin{example}\label{Omega-2}
Recall that $\Omega_2=KQ_2/\langle\alpha^2, \beta^2\rangle$ with
\begin{center}
$Q_2:\xymatrix@C=1cm{1 \ar[r]^{\mu}\ar@(dl,ul)^{\alpha}& 2 \ar@(ur,dr)^{\beta}}$.
\end{center}
It is first known as a $\tau$-tilting infinite gentle algebra in \cite[Theorem 1.1]{P-gentle}, and then is known to be minimal $\tau$-tilting infinite in \cite[Lemma 3.4]{W-two-point}. We look at the $g$-vectors of two-term silting complexes appearing in mutation order $2,1,2,1,2, \cdots$. Let $P_1$ and $P_2$ be the indecomposable projective $\Omega_2$-modules. We have
\begin{center}
$P_1=$\scalebox{0.8}{$\vcenter{\xymatrix@C=0.1cm@R=0.2cm{
&e_1\ar@{-}[dl]\ar@{-}[dr]&\\
\alpha\ar@{-}[d]&&\mu \ar@{-}[d] \\
\alpha\mu\ar@{-}[d]&& \mu\beta\\
\alpha\mu\beta&&}}
\simeq \vcenter{\xymatrix@C=0.2cm@R=0.2cm{
&1\ar@{-}[dl]\ar@{-}[dr]&\\
1\ar@{-}[d]&&2\ar@{-}[d] \\
2\ar@{-}[d]&& 2\\
2&&}}$} and $P_2=$\scalebox{0.8}{$\vcenter{\xymatrix@C=0.01cm@R=0.3cm{
e_2\ar@{-}[d]\\
\beta\\}}
\simeq \vcenter{\xymatrix@C=0.01cm@R=0.2cm{
2\ar@{-}[d]\\
2\\}}$}.
\end{center}
Similar to the Kronecker algebra case, any two-term silting complex in the mutation order $2,1,2,1,2, \cdots$ is tilting due to $\mathsf{Hom}_{\Omega_2}(P_1,P_2)=0$. As usual, we start with
\begin{center}
$\mu_{2}^-(\Omega_2)=\left [\begin{smallmatrix}
0\longrightarrow P_1\\
\oplus \\
P_2\overset{f}{\longrightarrow} P_1^{\oplus 2}
\end{smallmatrix}  \right ]$, where $f=\left ( \begin{smallmatrix}
\mu\\
\mu\beta
\end{smallmatrix} \right )$.
\end{center}
Set $X:=(0\longrightarrow P_1)$ and $Y:=(P_2\overset{f}{\longrightarrow} P_1^{\oplus 2})$. Then,
\begin{center}
$\mathsf{Hom}_{\mathcal{K}_{\Omega_2}}(Y,X)=0$, $\mathsf{Hom}_{\mathcal{K}_{\Omega_2}}(X,X)=\mathsf{span} \left \{(0,1), (0,\alpha) \right \}$,
\end{center}

\begin{center}
$\mathsf{Hom}_{\mathcal{K}_{\Omega_2}}(Y,Y)=\mathsf{span}\left \{\left (1, \left ( \begin{smallmatrix}
1&0\\
0&1
\end{smallmatrix} \right )\right ), \left (\beta, \left ( \begin{smallmatrix}
0&1\\
0&0
\end{smallmatrix} \right )\right )\right \}$,
\end{center}

\begin{center}
$\mathsf{Hom}_{\mathcal{K}_{\Omega_2}}(X,Y)=\mathsf{span} \left \{\left (0, \left ( \begin{smallmatrix}
1\\
0
\end{smallmatrix} \right )\right ), \left (0, \left ( \begin{smallmatrix}
0\\
1
\end{smallmatrix} \right )\right ), \left (0, \left ( \begin{smallmatrix}
\alpha\\
0
\end{smallmatrix} \right )\right ), \left (0, \left ( \begin{smallmatrix}
0\\
\alpha
\end{smallmatrix} \right )\right )\right \}$.
\end{center}
We define $a:=(0,\alpha)$, $b:=\left (\beta, \left ( \begin{smallmatrix}
0&1\\
0&0
\end{smallmatrix} \right )\right )$ and $c:=\left (0, \left ( \begin{smallmatrix}
0\\
1
\end{smallmatrix} \right )\right )$. Then,
\begin{center}
$a^2=b^2=0$, $bc=\left (0, \left ( \begin{smallmatrix}
1\\
0
\end{smallmatrix} \right )\right )$, $ca=\left (0, \left ( \begin{smallmatrix}
0\\
\alpha
\end{smallmatrix} \right )\right )$, $bca=\left (0, \left ( \begin{smallmatrix}
\alpha\\
0
\end{smallmatrix} \right )\right )$.
\end{center}
Thus, the endomorphism algebra $\mathsf{End}_{\mathcal{K}_{\Omega_2}}(\mu_{2}^-(\Omega_2))$ is isomorphic to $KQ/I$ with
\begin{center}
$Q: \xymatrix@C=1cm{1 \ar@(dl,ul)^{a}& 2\ar[l]_{c}\ar@(ur,dr)^{b}}$ bounded by $I=\langle a^2,b^2\rangle$,
\end{center}
which is exactly the opposite algebra $\Omega_2^{\text{op}}$ of $\Omega_2$. By Lemma \ref{lemma-endomorphism-algebra}, we find
\begin{center}
$\mathsf{End}_{\mathcal{K}_{\Omega_2}}(\mu_1^-(\mu_{2}^-(\Omega_2)))\simeq \mathsf{End}_{\mathcal{K}_{\Omega_2^\text{op}}}(\mu_{1}^-(\Omega_2^{\text{op}}))$.
\end{center}

One can check that the indecomposable projective $\Omega_2^{\text{op}}$-modules are displayed as
\begin{center}
\scalebox{0.8}{$\vcenter{\xymatrix@C=0.2cm@R=0.2cm{
1\ar@{-}[d]\\
1\\}}$} and \scalebox{0.8}{$\vcenter{\xymatrix@C=0.2cm@R=0.2cm{
&2\ar@{-}[dl]\ar@{-}[dr]&\\
2\ar@{-}[d]&&1\ar@{-}[d] \\
1\ar@{-}[d]&& 1\\
1&&}}$}.
\end{center}
Hence, the calculation of $\mu_{1}^-(\Omega_2^{\text{op}})$ is actually the same as that of $\mu_{2}^-(\Omega_2)$. By Lemma \ref{lemma-g-vector},
\begin{center}
$g(\mu_1^-(\mu_{2}^-(\Omega_2)))=\begin{smallmatrix}
(3,-2)\\
\oplus \\
(2,-1)
\end{smallmatrix}$.
\end{center}
Similar to Example \ref{example-kronecker-alg}, we also have
$\mathsf{End}_{\mathcal{K}_{\Omega_2^{\text{op}}}}(\mu_{1}^-(\Omega_2^{\text{op}}))\simeq (\Omega_2^{\text{op}})^{\text{op}}\simeq \Omega_2$ and
\begin{center}
$\mathsf{End}_{\mathcal{K}_{\Omega_2}}(\mu_2^-(\mu_1^-(\mu_{2}^-(\Omega_2))))
\simeq \mathsf{End}_{\mathcal{K}_{\Omega_2^{\text{op}}}}(\mu_2^-(\mu_{1}^-(\Omega_2^{\text{op}})))
\simeq \mathsf{End}_{\mathcal{K}_{\Omega_2}}(\mu_2^-(\Omega_2))
\simeq \Omega_2^{\text{op}}$.
\end{center}
The remaining part is the same as the part in Example \ref{example-kronecker-alg} and we omit the details.
\end{example}

\begin{remark}\label{Hasse-quiver-kronecker-alg}
Let $\Delta=\Delta_1$ or $\Omega_2$. It is easy to see that
\begin{center}
$\mu_{1}^-(\Delta)=\left [\begin{smallmatrix}
P_1\longrightarrow 0\\
\oplus \\
0 \longrightarrow P_2
\end{smallmatrix}  \right ]$ and $\mu_2^-(\mu_{1}^-(\Delta))=\left [\begin{smallmatrix}
P_1\longrightarrow 0\\
\oplus \\
P_2 \longrightarrow 0
\end{smallmatrix}  \right ]$.
\end{center}
There exists an algebra isomorphism $\sigma:\Delta^\text{op}\rightarrow \Delta$ satisfying $\sigma(e_1^\ast)=e_2$ and $\sigma(e_2^\ast)=e_1$. According to the anti-automorphism $S_\sigma$ introduced in Lemma \ref{lem-symmetry}, we find that the $g$-vectors for $\Delta$ are
\begin{center}
\begin{tikzpicture}[shorten >=1pt, auto, node distance=0cm,
   node_style/.style={font=},
   edge_style/.style={draw=black}]
\node[node_style] (v0) at (0,0) {$\begin{smallmatrix}
(1,0)\\
\oplus \\
(0,1)
\end{smallmatrix}$};
\node[node_style] (v1) at (12,0) {$\begin{smallmatrix}
(-1,0)\\
\oplus \\
(0,1)
\end{smallmatrix}$};
\node[node_style] (v) at (12,-4) {$\begin{smallmatrix}
(-1,0)\\
\oplus \\
(0,-1)
\end{smallmatrix}$};
\node[node_style] (v2) at (0,-2) {$\begin{smallmatrix}
(1,0)\\
\oplus \\
(2,-1)
\end{smallmatrix}$};
\node[node_style] (v21) at (0,-4) {$\begin{smallmatrix}
(3,-2)\\
\oplus \\
(2,-1)
\end{smallmatrix}$};
\node[node_style] (v212) at (2,-4) {$\begin{smallmatrix}
(3,-2)\\
\oplus \\
(4,-3)
\end{smallmatrix}$};
\node[node_style] (v2121) at (4,-4) {$\cdots$};

\node[node_style] (v212121) at (6,-4) {$\begin{smallmatrix}
(2,-3)\\
\oplus \\
(3,-4)
\end{smallmatrix}$};
\node[node_style] (v2121212) at (8,-4) {$\begin{smallmatrix}
(2,-3)\\
\oplus \\
(1,-2)
\end{smallmatrix}$};
\node[node_style] (v21212121) at (10,-4) {$\begin{smallmatrix}
(0,-1)\\
\oplus \\
(1,-2)
\end{smallmatrix}$};
\draw[->]  (v0) edge node{\ } (v1);
\draw[->]  (v1) edge node{\ } (v);
\draw[->]  (v0) edge node{\ } (v2);
\draw[->]  (v2) edge node{\ } (v21);
\draw[->]  (v21) edge node{\ } (v212);
\draw[->]  (v212) edge node{\ } (v2121);
\draw[->]  (v2121) edge node{\ } (v212121);
\draw[->]  (v212121) edge node{\ } (v2121212);
\draw[->]  (v2121212) edge node{\ } (v21212121);
\draw[->]  (v21212121) edge node{\ } (v);
\end{tikzpicture}.
\end{center}
It turns out that $\Delta_1$ and $\Omega_2$ share the same $g$-vectors.
\end{remark}

Recall that $\Lambda$ is a two-point algebra. Let $T\not\simeq \Lambda[1]$ be a two-term tilting complex in $\mathcal{K}_{\Lambda}$, which is obtained by iterated mutations from $\Lambda$. According to the form of the Hasse quiver $\mathcal{H}(\mathsf{2\text{-}silt}\ \Lambda)$ (\ref{the-form-of-Hasse-quiver}), we may assume, without loss of generality, that
\begin{center}
$T=\mu_2^-(\cdots (\mu_1^-(\mu_2^-(\mu_1^-(\Lambda)))))$.
\end{center}
Suppose that $\mu_1^-(T)$ is two-term (but $\mu_2^-(T)$ is not). We define $\Gamma:=\mathsf{End}_{\mathcal{K}_{\Lambda}}\ (T)$ and consider the following left mutation chains:
\begin{center}
$\vcenter{\xymatrix@C=0.5cm{\Lambda: &T\ar[r]\ar[d]^{F}&\mu_1^-(T)\ar[d]^{F}\\
&\Gamma \ar[r]&\mu_1^-(\Gamma)\ar[r]&\mu_2^-(\mu_1^-(\Gamma))\ar[r]&\mu_1^-(\mu_2^-(\mu_1^-(\Gamma)))\ar[r]&\cdots}}$.
\end{center}
We notice also from (\ref{the-form-of-Hasse-quiver}) that if $\Gamma$ is $\tau$-tilting infinite, then either the left mutation chain in mutation order $1,2,1,2,1,\cdots$ or the left mutation chain in mutation order $2,1,2,1,2,\cdots$ is an infinite chain.

\begin{lemma}\label{lemma-derived-infinite}
Under the above setting, $\Lambda$ is $\tau$-tilting infinite if the left mutation chain starting with $\Gamma$ in mutation order $1,2,1,2,1,\cdots$ is an infinite chain.
\end{lemma}
\begin{proof}
Since the triangle equivalence $F$ preserves the partial order $\geqslant$ and left mutations, it gives an isomorphism
\begin{center}
$\{U\in \mathsf{2\text{-}silt}\ \Lambda \mid T \geqslant U \geqslant \Lambda[1] \}$

$\overset{\sim}{\longrightarrow}
\{V\in \mathsf{2\text{-}silt}\ \Gamma \mid \Gamma \geqslant V \geqslant \mu_2^+(\cdots (\mu_1^+(\mu_2^+(\mu_1^+(\Gamma[1])))))\}$.
\end{center}
By the assumption on $\Gamma$, we conclude that the set $\mathsf{2\text{-}silt}\ \Lambda$ is infinite.
\end{proof}

\begin{proposition}\label{Omega-3}
Recall that $\Omega_3=KQ_3/\langle\beta_1^2,\beta_2^2,\beta_1\beta_2,\beta_2\beta_1\rangle$ with
\begin{center}
$Q_3: \xymatrix@C=1cm{1\ar[r]^{\mu}& 2 \ar@(ul,ur)^{\beta_1}\ar@(dl,dr)_{\beta_2}}$.
\end{center}
Then, $\Omega_3$ is minimal $\tau$-tilting infinite.
\end{proposition}
\begin{proof}
(1) We show that $\Omega_3$ is $\tau$-tilting infinite. Let $P_1$ and $P_2$ be the indecomposable projective $\Omega_3$-modules. Then,
\begin{center}
$P_1=$\scalebox{0.8}{$\vcenter{\xymatrix@C=0.1cm@R=0.2cm{
&e_1\ar@{-}[d]&\\
&\mu\ar@{-}[dl]\ar@{-}[dr]&\\
\mu\beta_1 && \mu\beta_2}}
\simeq \vcenter{\xymatrix@C=0.1cm@R=0.2cm{
&1\ar@{-}[d]&\\
&2\ar@{-}[dl]\ar@{-}[dr]&\\
2 && 2}}$} and $P_2=$\scalebox{0.8}{$\vcenter{\xymatrix@C=0.01cm@R=0.2cm{
&e_2\ar@{-}[dl]\ar@{-}[dr]&\\
\beta_1&&\beta_2\\}}
\simeq \vcenter{\xymatrix@C=0.01cm@R=0.2cm{
&2\ar@{-}[dl]\ar@{-}[dr]&\\
2&&2\\}}$}.
\end{center}
Obviously,
\begin{center}
$\begin{aligned}
\mathsf{Hom}_{\Omega_3}(P_1,P_1)&=1,\ \mathsf{Hom}_{\Omega_3}(P_1,P_2)=0,\\
\mathsf{Hom}_{\Omega_3}(P_2,P_2)&=\mathsf{span}\{1, \beta_1, \beta_2\},\\
\mathsf{Hom}_{\Omega_3}(P_2,P_1)&=\mathsf{span}\{\mu, \mu\beta_1, \mu\beta_2\}.
\end{aligned}$
\end{center}

We start with
\begin{center}
$\mu_{2}^-(\Omega_3)=\left [\begin{smallmatrix}
0\longrightarrow P_1\\
\oplus \\
P_2\overset{f}{\longrightarrow} P_1^{\oplus 3}
\end{smallmatrix}  \right ]$, where $f=\left ( \begin{smallmatrix}
\mu\\
\mu\beta_1\\
\mu\beta_2
\end{smallmatrix} \right )$.
\end{center}
This is obviously a tilting complex so that the assumption in Lemma \ref{lemma-derived-infinite} is satisfied. Let $X:=(0\longrightarrow P_1)$ and $Y:=(P_2\overset{f}{\longrightarrow} P_1^{\oplus 3})$. Then,
\begin{center}
$\mathsf{Hom}_{\mathcal{K}_{\Omega_3}}(Y,X)=0$, $\mathsf{Hom}_{\mathcal{K}_{\Omega_3}}(X,X)=\{(0,1)\}$,
\end{center}
\begin{center}
$\mathsf{Hom}_{\mathcal{K}_{\Omega_3}}(X,Y)=\mathsf{span} \left \{\left (0, \left ( \begin{smallmatrix}
1\\
0\\
0
\end{smallmatrix} \right )\right ), \left (0, \left ( \begin{smallmatrix}
0\\
1\\
0
\end{smallmatrix} \right )\right ), \left (0, \left ( \begin{smallmatrix}
0\\
0\\
1
\end{smallmatrix} \right )\right )\right \}$,
\end{center}
\begin{center}
$\mathsf{Hom}_{\mathcal{K}_{\Omega_3}}(Y,Y)=\mathsf{span} \left \{\left (1, \left ( \begin{smallmatrix}
1&0&0\\
0&1&0\\
0&0&1
\end{smallmatrix} \right )\right ), \left (\beta_1, \left ( \begin{smallmatrix}
0&1&0\\
0&0&0\\
0&0&0
\end{smallmatrix} \right )\right ), \left (\beta_2, \left ( \begin{smallmatrix}
0&0&1\\
0&0&0\\
0&0&0
\end{smallmatrix} \right )\right )\right \}$.
\end{center}
Set
\begin{center}
$b_1:=\left (\beta_1, \left ( \begin{smallmatrix}
0&1&0\\
0&0&0\\
0&0&0
\end{smallmatrix} \right )\right )$, $b_2:=\left (\beta_2, \left ( \begin{smallmatrix}
0&0&1\\
0&0&0\\
0&0&0
\end{smallmatrix} \right )\right )$, $c_1:=\left (0, \left ( \begin{smallmatrix}
0\\
1\\
0
\end{smallmatrix} \right )\right )$, $c_2:=\left (0, \left ( \begin{smallmatrix}
0\\
0\\
1
\end{smallmatrix} \right )\right )$,
\end{center}
and then,
\begin{center}
$b_1^2=b_2^2=b_1b_2=b_2b_1=b_1c_2=b_2c_1=0$, $b_1c_1=b_2c_2=\left (0, \left ( \begin{smallmatrix}
1\\
0\\
0
\end{smallmatrix} \right )\right )$.
\end{center}
It turns out that $\mathsf{End}_{\mathcal{K}_{\Omega_3}}(\mu_{2}^-(\Omega_3))$ is isomorphic to $\Delta_2:=KQ/I$ with
\begin{center}
$Q: \xymatrix@C=1cm{1 & 2\ar@<0.5ex>[l]^{c_1}\ar@<-0.5ex>[l]_{c_2} \ar@(dl,dr)_{b_2}\ar@(ul,ur)^{b_1}}$ and
$I=\langle b_1^2,b_2^2,b_1b_2,b_2b_1,b_1c_2,b_2c_1, b_1c_1-b_2c_2\rangle$.
\end{center}
By Example \ref{example-kronecker-alg}, one finds that the left mutation chain starting with $\Delta_2$ in mutation order $1,2,1,2,1,\cdots$ is an infinite chain. Thus, $\Omega_3$ is $\tau$-tilting infinite following Lemma \ref{lemma-derived-infinite}.

(2) Since the socle of $\Omega_3$ is $K\mu\beta_1\oplus K\mu\beta_2\oplus K\beta_1\oplus K\beta_2$, it suffices to consider
\begin{center}
$\overline{\Omega}_3:=\Omega_3/\langle \mu\beta_2\rangle$\ $(\simeq \Omega_3/\langle \mu\beta_1\rangle)$
\end{center}
for the minimality. Instead of calculating the two-term silting complexes in $\mathcal{K}_{\overline{\Omega}_3}$, we give the $g$-vectors for $\overline{\Omega}_3$. Note that $\overline{\Omega}_3$ is an algebra with radical cube zero, it is not difficult to find the $g$-vectors for $\overline{\Omega}_3$ by direct calculation (whether in $\tau$-tilting theory or silting theory). The $g$-vectors for $\overline{\Omega}_3$ are given by
\begin{center}
\begin{tikzpicture}[shorten >=1pt, auto, node distance=0cm,
   node_style/.style={font=},
   edge_style/.style={draw=black}]
\node[node_style] (v) at (0,0) {$\begin{smallmatrix}
(1,0)\\
\oplus \\
(0,1)
\end{smallmatrix}$};
\node[node_style] (v1) at (0,-2) {$\begin{smallmatrix}
(-1,0)\\
\oplus \\
(0,1)
\end{smallmatrix}$};
\node[node_style] (v2) at (2,0) {$\begin{smallmatrix}
(1,0)\\
\oplus \\
(2,-1)
\end{smallmatrix}$};
\node[node_style] (v21) at (4,0) {$\begin{smallmatrix}
(1,-1)\\
\oplus \\
(2,-1)
\end{smallmatrix}$};
\node[node_style] (v212) at (6,0) {$\begin{smallmatrix}
(1,-1)\\
\oplus \\
(0,-1)
\end{smallmatrix}$};
\node[node_style] (v0) at (6,-2) {$\begin{smallmatrix}
(-1,0)\\
\oplus \\
(0,-1)
\end{smallmatrix}$.};

\draw[->]  (v) edge node{ } (v1);
\draw[->]  (v1) edge node{ } (v0);
\draw[->]  (v) edge node{ } (v2);
\draw[->]  (v2) edge node{ } (v21);
\draw[->]  (v21) edge node{ } (v212);
\draw[->]  (v212) edge node{ } (v0);
\end{tikzpicture}
\end{center}
Thus, $\overline{\Omega}_3$ is $\tau$-tilting finite and $\Omega_3$ is minimal $\tau$-tilting infinite.
\end{proof}

\subsection{Two-point monomial algebras with radical cube zero}
\ \\
\vspace{-0.3cm}

We define $Q(m,n)$ as the quiver consisting of $m$ loops on vertex 1, $n$ loops on vertex 2, one arrow from 1 to 2 and one arrow from 2 to 1, i.e.,
\begin{center}
$Q(m,n):=\xymatrix@C=1.2cm{1\ar@<0.5ex>[r]^{\mu}\ar@(ul,ur)^{\alpha_1}\ar@(ld,lu)^{\vdots}\ar@(dr,dl)^{\alpha_m}
&2\ar@<0.5ex>[l]^{\nu}\ar@(ul,ur)^{\beta_1}\ar@(ru,rd)^{\vdots}\ar@(dr,dl)^{\beta_n}}$.
\end{center}

\begin{theorem}\label{result-rad-cube-zero}
Let $\Lambda$ be a two-point connected monomial algebra with radical cube zero. Then, $\Lambda$ is $\tau$-tilting finite if and only if it does not have one of the Kronecker algebra $\Delta_1$, $\Omega_3$ and $\Omega_3^{\emph{op}}$ as a quotient algebra.
\end{theorem}
\begin{proof}
If $\Lambda$ has one of $\Delta_1$, $\Omega_3$ and $\Omega_3^{\text{op}}$ as a quotient algebra, then $\Lambda$ is $\tau$-tilting infinite by Lemma \ref{lem-quotient-and-idempotent} (1). In the following, we assume that $\Lambda$ does not have one of $\Delta_1$, $\Omega_3$ and $\Omega_3^{\text{op}}$ as a quotient algebra. Then, the Gabriel quiver of $\Lambda$ is either $Q(m,n)$ or the subquiver of $Q(m,n)$ by deleting $\mu$ (or $\nu$), for some non-negative integers $m,n$. Moreover, the latter case is obviously covered by the former one.

We assume that $\Lambda=KQ(m,n)/I$ with a monomial ideal $I$. Note that the elements of $\mathsf{rad}^2\ KQ(m,n)$ are all linear combinations of the following elements:
\begin{enumerate}
  \item[(a)] $\mu\nu$, $\alpha_i\alpha_j$ for $i,j\in \{1,2,\cdots, m\}$;
  
  \item[(b)] $\nu\mu$, $\beta_i\beta_j$ for $i,j\in \{1,2,\cdots, n\}$;
  
  \item[(c)] $\mu\beta_s, \beta_s\nu, \nu\alpha_t, \alpha_t\mu$ for $s\in \{1,2,\cdots,n\}$ and $t\in \{1,2,\cdots,m\}$.
\end{enumerate}
Since $\Lambda$ is a radical cube zero algebra, the elements in $(a)$ and $(b)$ are included in the center of $\Lambda$. Let $J_{ab}$ be the two-sided ideal of $\Lambda$ generated by the elements in $(a)$ and $(b)$.
Then, we define $\widetilde{\Lambda}:=\Lambda/J_{ab}$.

Since $\Lambda$ does not have $\Omega_3$ or $\Omega_3^{\text{op}}$ as a quotient, there are at most one $\mu\beta_i\notin I$, at most one $\beta_j\nu\notin I$, at most one $\nu\alpha_k\notin I$ and at most one $\alpha_\ell\mu\notin I$, for $i, j\in \{1,2,\cdots,n\}$ and $k,\ell\in \{1,2,\cdots,m\}$. The existence of $i,j,k,\ell$ provides the fact that $\beta_i, \beta_j, \alpha_k, \alpha_\ell$ are not included in the center of $\widetilde{\Lambda}$. We set
\begin{center}
$J_c:=\left\langle\beta_s, \alpha_t \mid
\begin{matrix}
s\neq i,j \ \text{and}\ t\neq k, \ell \ \text{if}\  i,j,k,\ell \ \text{exist}\\
s\in \{1,2,\cdots,n\}, t\in \{1,2,\cdots,m\}
\end{matrix}
\right\rangle$,
\end{center}
this is a two-sided ideal of $\widetilde{\Lambda}$ generated by central elements. Then, we define $\Lambda':=\widetilde{\Lambda}/J_c$. By Lemma \ref{lem-center}, it is true that
\begin{center}
$\mathsf{2\text{-}silt}\ \Lambda\simeq \mathsf{2\text{-}silt}\ \widetilde{\Lambda}\simeq \mathsf{2\text{-}silt}\ \Lambda'$.
\end{center}

If all $i,j,k,\ell$ exist, then we can choose $(i,j,k,\ell)$ as $(1,1,1,1)$, $(1,2,1,2)$ or $(1,2,1,1)$, up to isomorphism and opposite algebras. Let $P_1$ and $P_2$ be the indecomposable projective $\Lambda'$-modules. We check the $\tau$-tilting finiteness of $\Lambda'$ case by case.
\begin{itemize}
\item In the case of $(1,1,1,1)$, the quiver of $\Lambda'$ is displayed as
\begin{center}
$\xymatrix@C=1cm{1\ar@(dl,ul)^{\alpha_1}\ar@<0.5ex>[r]^{\mu}&2\ar@<0.5ex>[l]^{\nu}\ar@(ur,dr)^{\beta_1}}$.
\end{center}
We have
\begin{center}
$P_1=$\scalebox{0.8}{$\vcenter{\xymatrix@C=0.1cm@R=0.2cm{
&e_1\ar@{-}[dl]\ar@{-}[dr]&\\
\alpha_1\ar@{-}[d] && \mu\ar@{-}[d]\\
\alpha_1\mu && \mu\beta_1}}
\simeq \vcenter{\xymatrix@C=0.1cm@R=0.2cm{
&1\ar@{-}[dl]\ar@{-}[dr]&\\
1\ar@{-}[d] && 2\ar@{-}[d]\\
2 && 2}}$} and $P_2=$\scalebox{0.8}{$\vcenter{\xymatrix@C=0.1cm@R=0.2cm{
&e_2\ar@{-}[dl]\ar@{-}[dr]&\\
\nu\ar@{-}[d] && \beta_1\ar@{-}[d]\\
\nu\alpha_1 && \beta_1\nu}}
\simeq \vcenter{\xymatrix@C=0.1cm@R=0.2cm{
&2\ar@{-}[dl]\ar@{-}[dr]&\\
1\ar@{-}[d] && 2\ar@{-}[d]\\
1 && 1}}$}.
\end{center}
By direct calculation, we find the following mutation chain $\mathsf{T}$ in $\mathcal{H}(\mathsf{2\text{-}silt}\ \Lambda')$,
\begin{center}
\begin{tikzpicture}[shorten >=1pt, auto, node distance=0cm,
   node_style/.style={font=},
   edge_style/.style={draw=black}]
\node[node_style] (v0) at (-0.5,0) {$\left [\begin{smallmatrix}
0\longrightarrow P_1\\
\oplus \\
0\longrightarrow P_2
\end{smallmatrix}  \right ]$};
\node[node_style] (v2) at (2,0) {$\left [\begin{smallmatrix}
0\longrightarrow P_1\\
\oplus \\
P_2\overset{f}{\longrightarrow} P_1^{\oplus 2}
\end{smallmatrix}  \right ]$};
\node[node_style] (v21) at (4.8,0) {$\left [\begin{smallmatrix}
P_2\overset{\mu}{\longrightarrow} P_1\\
\oplus \\
P_2\overset{f}{\longrightarrow} P_1^{\oplus 2}
\end{smallmatrix}  \right ]$,};

\node[node_style] (v212) at (7.6,0) {where $f=\left ( \begin{smallmatrix}
\mu\\
\mu\beta_1
\end{smallmatrix} \right )$.};

\draw[->]  (v0) edge node{\ } (v2);
\draw[->]  (v2) edge node{\ } (v21);
\end{tikzpicture}
\end{center}
In fact, we may explain the second step as follows. 
Set $X:=(0\longrightarrow P_1)$ and $Y:=(P_2\overset{f}{\longrightarrow} P_1^{\oplus 2})$. 
To compute $\mu_X^-(X\oplus Y)$, we take a triangle
\begin{center}
$\xymatrix@C=0.7cm{X\ar[r]^{\pi}&Y \ar[r]&\mathsf{cone}(\pi)\ar[r]&X[1]}$ with $\pi=\left ( 0,\left ( \begin{smallmatrix}
0\\
1
\end{smallmatrix} \right )\right )$.
\end{center}
On the one hand, if we compose $\pi$ with the endomorphism
\begin{center}
$\xymatrix@C=1.2cm@R=1cm{
Y: & P_2\ar[r]^{f}\ar[d]_{k_1e_2+k_2\beta_1} &P_1^{\oplus 2}\ar[d]^{\left ( \begin{smallmatrix}
k_1&k_2\\
0&k_1
\end{smallmatrix} \right )}\\
Y: & P_2\ar[r]^{f} &P_1^{\oplus 2}}$
\end{center}
for $k_1,k_2\in K$, then all elements of $\mathsf{Hom}_{\mathcal{K}_{\Lambda'}}(X,Y)$ are obtained. On the other hand, if $\left ( \begin{smallmatrix}
k_1&k_2\\
0&k_1
\end{smallmatrix} \right )\left ( \begin{smallmatrix}
0\\
1
\end{smallmatrix} \right )=\left ( \begin{smallmatrix}
0\\
1
\end{smallmatrix} \right )$, then $k_1=1$ and $k_2=0$. This implies that $\pi$ is a minimal left $\mathsf{add}(Y)$-approximation. Then, it is not difficult to find $\mathsf{cone}(\pi)=(P_2\overset{\mu}{\longrightarrow} P_1)$.

We observe that there exist two algebra isomorphisms $\sigma, \sigma':{\Lambda'}^\text{op}\rightarrow \Lambda'$ satisfying
\begin{center}
$\sigma(e_1^\ast)=e_2$, $\sigma(e_2^\ast)=e_1$ and $\sigma'(e_1^\ast)=e_1$, $\sigma'(e_2^\ast)=e_2$.
\end{center}
According to the anti-automorphism introduced in Lemma \ref{lem-symmetry}, there are other left mutation chains $S_\sigma(\mathsf{T})$, $S_{\sigma'}(\mathsf{T})$, $S_{\sigma'}(S_\sigma(\mathsf{T}))=S_{\sigma}(S_{\sigma'}(\mathsf{T}))$ in $\mathcal{H}(\mathsf{2\text{-}silt}\ \Lambda')$ whose $g$-vectors are displayed as
\begin{center}
$\begin{smallmatrix}
(1,-1)\\
\oplus \\
(1,-2)
\end{smallmatrix} \rightarrow \begin{smallmatrix}
(0,-1)\\
\oplus \\
(1,-2)
\end{smallmatrix} \rightarrow \begin{smallmatrix}
(0,-1)\\
\oplus \\
(-1,0)
\end{smallmatrix}$,\  $\begin{smallmatrix}
(-1,1)\\
\oplus \\
(-2,1)
\end{smallmatrix} \rightarrow \begin{smallmatrix}
(-1,0)\\
\oplus \\
(-2,1)
\end{smallmatrix} \rightarrow \begin{smallmatrix}
(-1,0)\\
\oplus \\
(0,-1)
\end{smallmatrix}$,\  $\begin{smallmatrix}
(0,1)\\
\oplus \\
(1,0)
\end{smallmatrix}\rightarrow \begin{smallmatrix}
(0,1)\\
\oplus \\
(-1,2)
\end{smallmatrix} \rightarrow \begin{smallmatrix}
(-1,1)\\
\oplus \\
(-1,2)
\end{smallmatrix}$,
\end{center}
respectively. Since each indecomposable two-term presilting complex in $\mathcal{K}_{\Lambda'}$ is a direct summand of exactly two two-term silting complexes in $\mathsf{2\text{-}silt}\ \Lambda'$, the $g$-vectors for $\Lambda'$ must be given by
\begin{center}
\begin{tikzpicture}[shorten >=1pt, auto, node distance=0cm,
   node_style/.style={font=},
   edge_style/.style={draw=black}]
\node[node_style] (v0) at (0,0) {$\begin{smallmatrix}
(1,0)\\
\oplus \\
(0,1)
\end{smallmatrix}$};
\node[node_style] (v1) at (2,0) {$\begin{smallmatrix}
(-1,2)\\
\oplus \\
(0,1)
\end{smallmatrix}$};
\node[node_style] (v12) at (4,0) {$\begin{smallmatrix}
(-1,2)\\
\oplus \\
(-1,1)
\end{smallmatrix}$};
\node[node_style] (v121) at (6,0) {$\begin{smallmatrix}
(-2,1)\\
\oplus \\
(-1,1)
\end{smallmatrix}$};
\node[node_style] (v1212) at (8,0) {$\begin{smallmatrix}
(-2,1)\\
\oplus \\
(-1,0)
\end{smallmatrix}$};
\node[node_style] (v) at (8,-2) {$\begin{smallmatrix}
(-1,0)\\
\oplus \\
(0,-1)
\end{smallmatrix}$};
\node[node_style] (v2) at (0,-2) {$\begin{smallmatrix}
(1,0)\\
\oplus \\
(2,-1)
\end{smallmatrix}$};
\node[node_style] (v21) at (2,-2) {$\begin{smallmatrix}
(1,-1)\\
\oplus \\
(2,-1)
\end{smallmatrix}$};
\node[node_style] (v212) at (4,-2) {$\begin{smallmatrix}
(1,-1)\\
\oplus \\
(1,-2)
\end{smallmatrix}$};
\node[node_style] (v2121) at (6,-2) {$\begin{smallmatrix}
(0,-1)\\
\oplus \\
(1,-2)
\end{smallmatrix}$};

\draw[->]  (v0) edge node{\ } (v1);
\draw[->]  (v1) edge node{\ } (v12);
\draw[->]  (v12) edge node{\ } (v121);
\draw[->]  (v121) edge node{\ } (v1212);
\draw[->]  (v1212) edge node{\ } (v);
\draw[->]  (v0) edge node{\ } (v2);
\draw[->]  (v2) edge node{\ } (v21);
\draw[->]  (v21) edge node{\ } (v212);
\draw[->]  (v212) edge node{\ } (v2121);
\draw[->]  (v2121) edge node{\ } (v);
\end{tikzpicture}.
\end{center}
This implies that $\Lambda'$ is $\tau$-tilting finite.

  \item In the case of $(1,2,1,2)$, we have
\begin{center}
$P_1=$\scalebox{0.8}{$\vcenter{\xymatrix@C=0.1cm@R=0.2cm{
&e_1\ar@{-}[dl]\ar@{-}[d]\ar@{-}[dr]&\\
\alpha_2\ar@{-}[d] &\alpha_1& \mu\ar@{-}[d]\\
\alpha_2\mu && \mu\beta_1}}
\simeq \vcenter{\xymatrix@C=0.2cm@R=0.2cm{
&1\ar@{-}[dl]\ar@{-}[d]\ar@{-}[dr]&\\
1\ar@{-}[d] &1& 2\ar@{-}[d]\\
2 && 2}}$} and $P_2=$\scalebox{0.8}{$\vcenter{\xymatrix@C=0.1cm@R=0.2cm{
&e_2\ar@{-}[dl]\ar@{-}[d]\ar@{-}[dr]&\\
\nu\ar@{-}[d] &\beta_1& \beta_2\ar@{-}[d]\\
\nu\alpha_1 && \beta_2\nu}}
\simeq \vcenter{\xymatrix@C=0.2cm@R=0.2cm{
&2\ar@{-}[dl]\ar@{-}[d]\ar@{-}[dr]&\\
1\ar@{-}[d] &2& 2\ar@{-}[d]\\
1 && 1}}$}.
\end{center}

\item
In the case of $(1,2,1,1)$, we have
\begin{center}
$P_1=$\scalebox{0.8}{$\vcenter{\xymatrix@C=0.1cm@R=0.2cm{
&e_1\ar@{-}[dl] \ar@{-}[dr]&\\
\alpha_2\ar@{-}[d] & & \mu\ar@{-}[d]\\
\alpha_2\mu && \mu\beta_1}}
\simeq \vcenter{\xymatrix@C=0.2cm@R=0.2cm{
&1\ar@{-}[dl] \ar@{-}[dr]&\\
1\ar@{-}[d] & & 2\ar@{-}[d]\\
2 && 2}}$} and $P_2=$\scalebox{0.8}{$\vcenter{\xymatrix@C=0.1cm@R=0.2cm{
&e_2\ar@{-}[dl]\ar@{-}[d]\ar@{-}[dr]&\\
\nu\ar@{-}[d] &\beta_1& \beta_2\ar@{-}[d]\\
\nu\alpha_1 && \beta_2\nu}}
\simeq \vcenter{\xymatrix@C=0.2cm@R=0.2cm{
&2\ar@{-}[dl]\ar@{-}[d]\ar@{-}[dr]&\\
1\ar@{-}[d] &2& 2\ar@{-}[d]\\
1 && 1}}$}.
\end{center}
Then, it is not difficult to check that both of them share the same $g$-vectors with $\Lambda'$ in the case of $(1,1,1,1)$. Thus, $\Lambda'$ is $\tau$-tilting finite in both these two cases.
\end{itemize}

If one of $i,j,k,\ell$ does not exist, then $\Lambda'$ is a quotient algebra of one of the above three cases. Hence, $\Lambda'$ is also $\tau$-tilting finite in this case.
\end{proof}

\subsection{Two-point monomial algebras with $Q_1$ and $Q_2$}
\ \\
\vspace{-0.3cm}

Recall that $\Omega_1=KQ_1/\langle\beta^4\rangle$ with
\begin{center}
$Q_1: \xymatrix@C=1cm{1 \ar[r]^{\mu}& 2\ar@(ur,dr)^{\beta } }$.
\end{center}

\begin{proposition}\label{Omega-1}
$\Omega_1$ is minimal $\tau$-tilting infinite.
\end{proposition}
\begin{proof}
Let $P_1$ and $P_2$ be the indecomposable projective $\Omega_1$-modules. We have
\begin{center}
$P_1=$\scalebox{0.8}{$\vcenter{\xymatrix@C=0.1cm@R=0.2cm{
e_1\ar@{-}[d]\\
\mu\ar@{-}[d]\\
\mu\beta\ar@{-}[d]\\
\mu\beta^2\ar@{-}[d]\\
\mu\beta^3}}
\simeq \vcenter{\xymatrix@C=0.1cm@R=0.2cm{
1\ar@{-}[d]\\
2\ar@{-}[d]\\
2\ar@{-}[d]\\
2\ar@{-}[d]\\
2}}$} and $P_2=$\scalebox{0.8}{$\vcenter{\xymatrix@C=0.1cm@R=0.2cm{
e_2\ar@{-}[d]\\
\beta\ar@{-}[d]\\
\beta^2\ar@{-}[d]\\
\beta^3}}
\simeq \vcenter{\xymatrix@C=0.1cm@R=0.2cm{
2\ar@{-}[d]\\
2\ar@{-}[d]\\
2\ar@{-}[d]\\
2}}$}.
\end{center}
It is easy to check that
\begin{center}
$\begin{aligned}
\mathsf{Hom}_{\Omega_1}(P_1,P_1)&=1,\ \mathsf{Hom}_{\Omega_1}(P_1,P_2)=0,\\
\mathsf{Hom}_{\Omega_1}(P_2,P_2)&=\mathsf{span}\{1, \beta, \beta^2, \beta^3\},\\
\mathsf{Hom}_{\Omega_1}(P_2,P_1)&=\mathsf{span}\{\mu, \mu\beta, \mu\beta^2, \mu\beta^3 \}.
\end{aligned}$
\end{center}

If one takes left mutation starting with $\Omega_1$ in mutation order $1,2,1,2,1, \cdots$, one may find that the mutation chain will reach $\Omega_1[1]$ quickly, i.e.,
\begin{center}
$\mu_{1}^-(\Omega_1)=\left [\begin{smallmatrix}
P_1\longrightarrow 0\\
\oplus \\
0 \longrightarrow P_2
\end{smallmatrix}  \right ]$ and $\mu_2^-(\mu_{1}^-(\Omega_1))=\left [\begin{smallmatrix}
P_1\longrightarrow 0\\
\oplus \\
P_2 \longrightarrow 0
\end{smallmatrix}  \right ]$.
\end{center}
We consider the left mutation chain starting with $\Omega_1$ in mutation order $2,1,2,1,2,\cdots$. In this case, we find that any two-term silting complex is tilting since $\mathsf{Hom}_{\Omega_1}(P_1,P_2)=0$ (see Remark \ref{silt-is-tilt}). In the beginning, we have
\begin{center}
$\mu_{2}^-(\Omega_1)=\left [\begin{smallmatrix}
0\longrightarrow P_1\\
\oplus \\
P_2\overset{f}{\longrightarrow} P_1^{\oplus 4}
\end{smallmatrix}  \right ]$ with $f=\left ( \begin{smallmatrix}
\mu\\
\mu\beta\\
\mu\beta^2\\
\mu\beta^3
\end{smallmatrix} \right )$.
\end{center}
We calculate the left mutations starting with $\mu_{2}^-(\Omega_1)$ as follows.

(1) Let $X:=(0\longrightarrow P_1)$ and $Y:=(P_2\overset{f}{\longrightarrow} P_1^{\oplus 4})$. Then,
\begin{center}
$\mathsf{Hom}_{\mathcal{K}_{\Omega_1}}(Y,X)=0$, $\mathsf{Hom}_{\mathcal{K}_{\Omega_1}}(X,X)=\{(0,1)\}$,
\end{center}

\begin{center}
$\begin{aligned}\mathsf{Hom}_{\mathcal{K}_{\Omega_1}}(Y,Y)=\mathsf{span}
&\left\{\begin{matrix}
\left (1, \left ( \begin{smallmatrix}
1&0&0&0\\
0&1&0&0\\
0&0&1&0\\
0&0&0&1
\end{smallmatrix} \right )\right ), \left (\beta, \left ( \begin{smallmatrix}
0&1&0&0\\
0&0&1&0\\
0&0&0&1\\
0&0&0&0
\end{smallmatrix} \right )\right ),
\end{matrix}\right. \\
&\ \ \left.\begin{matrix}
\left (\beta^2, \left ( \begin{smallmatrix}
0&0&1&0\\
0&0&0&1\\
0&0&0&0\\
0&0&0&0
\end{smallmatrix} \right )\right ), \left (\beta^3, \left ( \begin{smallmatrix}
0&0&0&1\\
0&0&0&0\\
0&0&0&0\\
0&0&0&0
\end{smallmatrix} \right )\right )
\end{matrix}\right\},
\end{aligned}$
\end{center}

\begin{center}
$\mathsf{Hom}_{\mathcal{K}_{\Omega_1}}(X,Y)=\mathsf{span} \left \{\left (0, \left ( \begin{smallmatrix}
1\\
0\\
0\\
0
\end{smallmatrix} \right )\right ), \left (0, \left ( \begin{smallmatrix}
0\\
1\\
0\\
0
\end{smallmatrix} \right )\right ), \left (0, \left ( \begin{smallmatrix}
0\\
0\\
1\\
0
\end{smallmatrix} \right )\right ), \left (0, \left ( \begin{smallmatrix}
0\\
0\\
0\\
1
\end{smallmatrix} \right )\right )\right \}$.
\end{center}
Set
\begin{center}
$a:=\left (0, \left ( \begin{smallmatrix}
0\\
0\\
0\\
1
\end{smallmatrix} \right )\right )$ and $b:=\left (\beta, \left ( \begin{smallmatrix}
0&1&0&0\\
0&0&1&0\\
0&0&0&1\\
0&0&0&0
\end{smallmatrix} \right )\right )$.
\end{center}
Then,
\begin{center}
$b^2=\left (\beta^2, \left ( \begin{smallmatrix}
0&0&1&0\\
0&0&0&1\\
0&0&0&0\\
0&0&0&0
\end{smallmatrix} \right )\right )$, $b^3=\left (\beta^3, \left ( \begin{smallmatrix}
0&0&0&1\\
0&0&0&0\\
0&0&0&0\\
0&0&0&0
\end{smallmatrix} \right )\right )$, $b^4=0$,
\end{center}

\begin{center}
$ba=\left (0, \left ( \begin{smallmatrix}
0\\
0\\
1\\
0
\end{smallmatrix} \right )\right )$, $b^2a=\left (0, \left ( \begin{smallmatrix}
0\\
1\\
0\\
0
\end{smallmatrix} \right )\right )$, $b^3a=\left (0, \left ( \begin{smallmatrix}
1\\
0\\
0\\
0
\end{smallmatrix} \right )\right )$.
\end{center}
This implies that $\mathsf{End}_{\mathcal{K}_{\Omega_1}}(\mu_{2}^-(\Omega_1))$ is isomorphic to $KQ/I$ with
\begin{equation}\label{oppsite-omega-1}
Q:\xymatrix@C=1cm{1 & 2\ar[l]_{a}\ar@(ur,dr)^{b}}\ \text{and}\ I=\langle b^4\rangle,
\end{equation}
which is exactly the opposite algebra $\Omega_1^{\text{op}}$ of $\Omega_1$. By Lemma \ref{lemma-endomorphism-algebra}, we have
\begin{center}
$\mathsf{End}_{\mathcal{K}_{\Omega_1}}(\mu_1^-(\mu_{2}^-(\Omega_1)))\simeq
\mathsf{End}_{\mathcal{K}_{\Omega_1^{\text{op}}}}(\mu_{1}^-(\Omega_1^{\text{op}}))$.
\end{center}

(2) Let $R_1$ and $R_2$ be the indecomposable projective $\Omega_1^{\text{op}}$-modules. Then,
\begin{center}
$R_1=e_1\simeq 1$ and $R_2=$\scalebox{0.8}{$\vcenter{\xymatrix@C=0.1cm@R=0.2cm{
&e_2\ar@{-}[dl]\ar@{-}[dr]&&&\\
a&&b\ar@{-}[dl]\ar@{-}[dr]&&\\
&ba&&b^2\ar@{-}[dl]\ar@{-}[dr]&\\
&&b^2a&&b^3\ar@{-}[dl]\\
&&&b^3a& }}$} $\simeq$ \scalebox{0.7}{$\vcenter{\xymatrix@C=0.2cm@R=0.3cm{
&2\ar@{-}[dl]\ar@{-}[dr]&&&\\
1&&2\ar@{-}[dl]\ar@{-}[dr]&&\\
&1&&2\ar@{-}[dl]\ar@{-}[dr]&\\
&&1&&2\ar@{-}[dl]\\
&&&1& }}$},
\end{center}
where we use the symbols in (\ref{oppsite-omega-1}). One may easily check that
\begin{center}
$\mu_{1}^-(\Omega_1^{\text{op}})=\left [\begin{smallmatrix}
R_1\overset{a}{\longrightarrow} R_2\\
\oplus \\
0\longrightarrow R_2
\end{smallmatrix}\right ]$.
\end{center}
By Lemma \ref{lemma-g-vector}, we have
\begin{center}
$g(\mu_1^-(\mu_{2}^-(\Omega_1)))=\begin{smallmatrix}
(3,-1)\\
\oplus \\
(4,-1)
\end{smallmatrix}$.
\end{center}

(3) Let $X:=(R_1\overset{a}{\longrightarrow} R_2)$ and $Y:=(0\longrightarrow R_2)$. Then,
\begin{center}
$\mathsf{Hom}_{\mathcal{K}_{\Omega_1^{\text{op}}}}(X,Y)=0$, $\mathsf{Hom}_{\mathcal{K}_{\Omega_1^{\text{op}}}}(X,X)=\{(1,1)\}$,
\end{center}

\begin{center}
$\mathsf{Hom}_{\mathcal{K}_{\Omega_1^{\text{op}}}}(Y,Y)=\mathsf{span} \left \{(0,1), (0, b),(0, b^2),(0, b^3)\right \}$,
\end{center}

\begin{center}
$\mathsf{Hom}_{\mathcal{K}_{\Omega_1^{\text{op}}}}(Y,X)=\mathsf{span} \left \{(0,1), (0, b),(0, b^2),(0, b^3)\right \}$.
\end{center}
It is not difficult to find that $\mathsf{End}_{\mathcal{K}_{\Omega_1^{\text{op}}}}(\mu_{1}^-(\Omega_1^{\text{op}}))\simeq \Omega_1$. By Lemma \ref{lemma-endomorphism-algebra},
\begin{center}
$\begin{aligned}\mathsf{End}_{\mathcal{K}_{\Omega_1}}(\mu_2^-(\mu_1^-(\mu_{2}^-(\Omega_1))))
&\simeq \mathsf{End}_{\mathcal{K}_{\Omega_1^{\text{op}}}}(\mu_2^-(\mu_{1}^-(\Omega_1^{\text{op}})))\\
&\simeq \mathsf{End}_{\mathcal{K}_{\Omega_1}}(\mu_2^-(\Omega_1))\\
&\simeq \Omega_1^{\text{op}}
\end{aligned}$.
\end{center}
By Lemma \ref{lemma-g-vector} and the calculation of $g(\mu_{2}^-(\Omega_1))$, we find
\begin{center}
$g(\mu_2^-(\mu_1^-(\mu_{2}^-(\Omega_1))))=\begin{smallmatrix}
(3,-1)\\
\oplus \\
(8,-3)
\end{smallmatrix}$.
\end{center}

By repeating the above procedures (1), (2) and (3), we obtain a mutation chain
\begin{center}
$\Omega_1=:T_0\longrightarrow T_1\longrightarrow T_2\longrightarrow T_3\longrightarrow \cdots  \in \mathcal{H}(\mathsf{2\text{-}silt}\ \Omega_1)$,
\end{center}
where $T_i$ is always tilting, $\mathsf{End}_{\mathcal{K}_{\Omega_1}}(T_{2n})\simeq \Omega_1$ and $\mathsf{End}_{\mathcal{K}_{\Omega_1}}(T_{2n+1})\simeq \Omega_1^{\text{op}}$. Then, the chain of $g$-vectors of $T_i$ is given by
\begin{center}
$\begin{smallmatrix}
(1,0)\\
\oplus \\
(0,1)
\end{smallmatrix}\rightarrow \begin{smallmatrix}
(1,0)\\
\oplus \\
(4,-1)
\end{smallmatrix}\rightarrow \begin{smallmatrix}
(3,-1)\\
\oplus \\
(4,-1)
\end{smallmatrix} \rightarrow \begin{smallmatrix}
(3,-1)\\
\oplus \\
(8,-3)
\end{smallmatrix} \rightarrow \cdots\rightarrow g(T_{2n})\rightarrow g(T_{2n+1})\rightarrow \cdots$.
\end{center}
The $g$-vector $g(T_{2n})$ is computed by $\left ( \begin{smallmatrix}
3& -1 \\
4 & -1
\end{smallmatrix} \right )^n$, so that
\begin{center}
$g(T_{2n})=\begin{smallmatrix}
(2n+1, -n )\\
\oplus \\
(4n, -2n+1)
\end{smallmatrix}$ and $g(T_{2n+1})=\begin{smallmatrix}
(2n+1, -n )\\
\oplus \\
(4n+4, -2n-1)
\end{smallmatrix}$.
\end{center}
It is obvious that the chain of $g$-vectors cannot reach $(-1,0)\oplus (0,-1)$. This implies that $\mathcal{H}(\mathsf{2\text{-}silt}\ \Omega_1)$ contains an infinite chain and hence, $\Omega_1$ is $\tau$-tilting infinite.

Lastly, we show the minimality. We should consider the quotient of $\Omega_1$ modulo the ideal generated by $\mu\beta^3$, because the socle of $\Omega_1$ is $K\mu\beta^3\oplus K\beta^3$. But we notice that $\beta^3$ is included in the center of the quotient. By taking Lemma \ref{lem-center} into our consideration, it suffices to focus on
\begin{center}
$\overline{\Omega}_1:=\Omega_1/\langle \beta^3\rangle$.
\end{center}
It is shown by \cite[Lemma 3.3]{W-two-point} that the $g$-vectors for $\overline{\Omega}_1$ are
\begin{center}
\begin{tikzpicture}[shorten >=1pt, auto, node distance=0cm,
   node_style/.style={font=},
   edge_style/.style={draw=black}]
\node[node_style] (v) at (-2,0) {$\begin{smallmatrix}
(1,0)\\
\oplus \\
(0,1)
\end{smallmatrix}$};
\node[node_style] (v3) at (0,0) {$\begin{smallmatrix}
(1,0)\\
\oplus \\
(3,-1)
\end{smallmatrix}$};
\node[node_style] (v1) at (-2,-2) {$\begin{smallmatrix}
(-1,0)\\
\oplus \\
(0,1)
\end{smallmatrix}$};
\node[node_style] (v2) at (2,0) {$\begin{smallmatrix}
(2,-1)\\
\oplus \\
(3,-1)
\end{smallmatrix}$};
\node[node_style] (v21) at (4,0) {$\begin{smallmatrix}
(2,-1)\\
\oplus \\
(3,-2)
\end{smallmatrix}$};
\node[node_style] (v212) at (6,0) {$\begin{smallmatrix}
(1,-1)\\
\oplus \\
(3,-2)
\end{smallmatrix}$};
\node[node_style] (v2121) at (8,0) {$\begin{smallmatrix}
(1,-1)\\
\oplus \\
(0,-1)
\end{smallmatrix}$};
\node[node_style] (v0) at (8,-2) {$\begin{smallmatrix}
(-1,0)\\
\oplus \\
(0,-1)
\end{smallmatrix}$,};

\draw[->]  (v) edge node{ } (v1);
\draw[->]  (v1) edge node{ } (v0);
\draw[->]  (v) edge node{ } (v3);
\draw[->]  (v3) edge node{ } (v2);
\draw[->]  (v2) edge node{ } (v21);
\draw[->]  (v21) edge node{ } (v212);
\draw[->]  (v212) edge node{ } (v2121);
\draw[->]  (v2121) edge node{ } (v0);
\end{tikzpicture}
\end{center}
which implies that $\overline{\Omega}_1$ is $\tau$-tilting finite.
\end{proof}

\begin{theorem}\label{result-Q1}
Let $\Lambda=KQ_1/I$ be an algebra with an admissible ideal $I$. Then, $\Lambda$ is $\tau$-tilting finite if and only if it does not have $\Omega_1$ as a quotient.
\end{theorem}
\begin{proof}
If $\Lambda$ has $\Omega_1$ as a quotient algebra, then $\Lambda$ is $\tau$-tilting infinite by Lemma \ref{lem-quotient-and-idempotent} (1). If $\Lambda$ does not have $\Omega_1$ as a quotient, we show that either $\Lambda$ is a quotient of $\overline{\Omega}_1$ or $\beta^3$ is a central element of $\Lambda$. In the latter case, $\Lambda/\langle \beta^3\rangle$ is a quotient of $\overline{\Omega}_1$.

Let $r=\text{min}\{i\in \mathbb{N}\mid \beta^i\in I\}$. The indecomposable projective modules of $\Lambda$ are
\begin{center}
$P_1=$\scalebox{0.8}{$\vcenter{\xymatrix@C=0.1cm@R=0.2cm{
e_1\ar@{-}[d]\\
\mu\ar@{-}[d]\\
\mu\beta\ar@{-}[d]\\
\vdots\\
\mu\beta^{r-1}}}$} and $P_2=$\scalebox{0.8}{$\vcenter{\xymatrix@C=0.1cm@R=0.2cm{
e_2\ar@{-}[d]\\
\beta\ar@{-}[d]\\
\vdots\\
\beta^{r-1}}}$}.
\end{center}
If $r=2$ or $3$, then $\beta^3\in I$ and $\Lambda$ is a quotient of $\overline{\Omega}_1$. Suppose that $r\geqslant 4$. If $\mu\beta^3\not\in I$, then $\Omega_1$ is a quotient of $\Lambda$, a contradiction. So we must have $\mu\beta^3\in I$. In this case, $\beta^3$ is a central element of $\Lambda$ and $\Lambda/\langle \beta^3\rangle\simeq \overline{\Omega}_1$ if $\mu\beta^2\notin I$, $\Lambda/\langle \beta^3\rangle$ is a quotient of $\overline{\Omega}_1$ if $\mu\beta^2\in I$. Therefore, $\Lambda$ is $\tau$-tilting finite if it does not have $\Omega_1$ as a quotient.
\end{proof}

Recall that $\Omega_2=KQ_2/\langle\alpha^2, \beta^2\rangle$ with
\begin{center}
$Q_2:\xymatrix@C=1cm{1 \ar[r]^{\mu}\ar@(dl,ul)^{\alpha}& 2 \ar@(ur,dr)^{\beta}}$.
\end{center}
It is seen from \cite[Lemma 3.4]{W-two-point} that $\Omega_2$ is minimal $\tau$-tilting infinite.

\begin{proposition}\label{Delta-3}
Let $\Delta_3:=KQ_2/\langle \alpha^3, \beta^3, \alpha\mu\beta \rangle$. Then, $\Delta_3$ is $\tau$-tilting finite.
\end{proposition}
\begin{proof}
We determine the $g$-vectors for $\Delta_3$ by using the calculation of $\mathcal{H}(\mathsf{2\text{-}silt}\ \overline{\Omega}_1)$ and Lemma \ref{lem-symmetry}. Recall that
\begin{center}
$\overline{\Omega}_1=\Omega_1/\langle \beta^3\rangle(\simeq KQ_2/\langle \alpha, \beta^3\rangle)$.
\end{center}
Let $R_1$ and $R_2$ be the indecomposable projective $\overline{\Omega}_1$-modules. Then,
\begin{center}
$R_1=$\scalebox{0.8}{$\vcenter{\xymatrix@C=0.1cm@R=0.2cm{
e_1\ar@{-}[d]\\
\mu\ar@{-}[d]\\
\mu\beta\ar@{-}[d]\\
\mu\beta^2}}
\simeq \vcenter{\xymatrix@C=0.1cm@R=0.2cm{
1\ar@{-}[d]\\
2\ar@{-}[d]\\
2\ar@{-}[d]\\
2}}$} and $R_2=$\scalebox{0.8}{$\vcenter{\xymatrix@C=0.1cm@R=0.2cm{
e_2\ar@{-}[d]\\
\beta\ar@{-}[d]\\
\beta^2}}
\simeq \vcenter{\xymatrix@C=0.1cm@R=0.2cm{
2\ar@{-}[d]\\
2\ar@{-}[d]\\
2}}$}.
\end{center}
It is shown in \cite[Lemma 3.3]{W-two-point} that we have the following chain in $\mathcal{H}(\mathsf{2\text{-}silt}\ \overline{\Omega}_1)$,
\begin{center}
\begin{tikzpicture}[shorten >=1pt, auto, node distance=0cm,
   node_style/.style={font=},
   edge_style/.style={draw=black}]
\node[node_style] (v0) at (-0.5,0) {$\left [\begin{smallmatrix}
0\longrightarrow R_1\\
\oplus \\
0\longrightarrow R_2
\end{smallmatrix}  \right ]$};
\node[node_style] (v2) at (2,0) {$\left [\begin{smallmatrix}
0\longrightarrow R_1\\
\oplus \\
R_2\overset{f_1}{\longrightarrow} R_1^{\oplus 3}
\end{smallmatrix}  \right ]$};
\node[node_style] (v21) at (4.8,0) {$\left [\begin{smallmatrix}
R_2\overset{f_2}{\longrightarrow} R_1^{\oplus 2}\\
\oplus \\
R_2\overset{f_1}{\longrightarrow} R_1^{\oplus 3}
\end{smallmatrix}  \right ]$};
\node[node_style] (v212) at (7.7,0) {$\left [\begin{smallmatrix}
R_2\overset{f_2}{\longrightarrow} R_1^{\oplus 2}\\
\oplus \\
R_2^{\oplus 2}\overset{f_3}{\longrightarrow} R_1^{\oplus 3}
\end{smallmatrix}  \right ]$};
\node[node_style] (v2121) at (10.7,0) {$\left [\begin{smallmatrix}
R_2\overset{\mu}{\longrightarrow} R_1\\
\oplus \\
R_2^{\oplus 2}\overset{f_3}{\longrightarrow} R_1^{\oplus 3}
\end{smallmatrix}  \right ]$};

\draw[->]  (v0) edge node{\ } (v2);
\draw[->]  (v2) edge node{\ } (v21);
\draw[->]  (v21) edge node{\ } (v212);
\draw[->]  (v212) edge node{\ } (v2121);
\end{tikzpicture}
\end{center}
where
\begin{center}
$f_1=\left ( \begin{smallmatrix}
\mu\\
\mu\beta\\
\mu\beta^2
\end{smallmatrix} \right )$, $f_2=\left ( \begin{smallmatrix}
\mu\\
\mu\beta
\end{smallmatrix} \right )$, $f_3=\left ( \begin{smallmatrix}
\mu&0\\
-\mu\beta &\mu\\
0&\mu\beta
\end{smallmatrix} \right )$.
\end{center}

Second, the indecomposable projective $\Delta_3$-modules are
\begin{center}
$P_1=$\scalebox{0.8}{$\vcenter{\xymatrix@C=0.1cm@R=0.2cm{
&e_1\ar@{-}[dl]\ar@{-}[dr]&&&\\
\mu\ar@{-}[d]&&\alpha\ar@{-}[dl]\ar@{-}[dr]&&\\
\mu\beta\ar@{-}[d]&\alpha\mu&&\alpha^2\ar@{-}[dl]&\\
\mu\beta^2&&\alpha^2\mu&& }}$} $\simeq$ \scalebox{0.7}{$\vcenter{\xymatrix@C=0.2cm@R=0.2cm{
&1\ar@{-}[dl]\ar@{-}[dr]&&&\\
2\ar@{-}[d]&&1\ar@{-}[dl]\ar@{-}[dr]&&\\
2\ar@{-}[d]&2&&1\ar@{-}[dl]&\\
2&&2&& }}$} and $P_2=$\scalebox{0.8}{$\vcenter{\xymatrix@C=0.1cm@R=0.2cm{
e_2\ar@{-}[d]\\
\beta\ar@{-}[d]\\
\beta^2}}
\simeq \vcenter{\xymatrix@C=0.1cm@R=0.2cm{
2\ar@{-}[d]\\
2\ar@{-}[d]\\
2}}$}.
\end{center}
Comparing with the two-term silting complexes in  $\mathcal{K}_{\overline{\Omega}_1}$, one finds that
\begin{center}
$(P_2\overset{\mu}{\longrightarrow} P_1)$, $(P_2\overset{f_1}{\longrightarrow} P_1^{\oplus 3})$, $(P_2\overset{f_2}{\longrightarrow} P_1^{\oplus 2})$, $(P_2^{\oplus 2}\overset{f_3}{\longrightarrow} P_1^{\oplus 3})$,
\end{center}
are two-term presilting complexes in $\mathcal{K}_{\Delta_3}$. 
This gives us all the necessary data to determine the $g$-vectors for $\Delta_3$.
We obtain a left mutation chain $\mathsf{T}$ in $\mathcal{H}(\mathsf{2\text{-}silt}\ \Delta_3)$ whose $g$-vectors are displayed as
\begin{center}
$\begin{smallmatrix}
(1,0)\\
\oplus \\
(0,1)
\end{smallmatrix}\longrightarrow \begin{smallmatrix}
(1,0)\\
\oplus \\
(3,-1)
\end{smallmatrix}\longrightarrow \begin{smallmatrix}
(2,-1)\\
\oplus \\
(3,-1)
\end{smallmatrix} \longrightarrow \begin{smallmatrix}
(2,-1)\\
\oplus \\
(3,-2)
\end{smallmatrix} \longrightarrow \begin{smallmatrix}
(1,-1)\\
\oplus \\
(3,-2)
\end{smallmatrix}$.
\end{center}

Third, we observe that there exists an algebra isomorphism $\sigma:\Delta_3^\text{op}\rightarrow \Delta_3$ satisfying
\begin{center}
$\sigma(e_1^\ast)=e_2$ and $\sigma(e_2^\ast)=e_1$.
\end{center}
According to the anti-automorphism $S_\sigma$ introduced in Lemma \ref{lem-symmetry}, there is another left mutation chain $S_\sigma(\mathsf{T})$ in $\mathcal{H}(\mathsf{2\text{-}silt}\ \Delta_3)$ whose $g$-vectors are displayed as
\begin{center}
$\begin{smallmatrix}
(1,-1)\\
\oplus \\
(2,-3)
\end{smallmatrix}\longrightarrow \begin{smallmatrix}
(1,-2)\\
\oplus \\
(2,-3)
\end{smallmatrix}\longrightarrow \begin{smallmatrix}
(1,-2)\\
\oplus \\
(1,-3)
\end{smallmatrix} \longrightarrow \begin{smallmatrix}
(0,-1)\\
\oplus \\
(1,-3)
\end{smallmatrix} \longrightarrow \begin{smallmatrix}
(0,-1)\\
\oplus \\
(-1,0)
\end{smallmatrix}$.
\end{center}
Since each indecomposable two-term presilting complex in $\mathcal{K}_{\Delta_3}$ is a direct summand of exactly two two-term silting
complexes in $\mathsf{2\text{-}silt}\ \Delta_3$, we must have an arrow in $\mathcal{H}(\mathsf{2\text{-}silt}\ \Delta_3)$ which is determined by
\begin{center}
$\begin{smallmatrix}
(1,-1)\\
\oplus \\
(3,-2)
\end{smallmatrix}\longrightarrow \begin{smallmatrix}
(1,-1)\\
\oplus \\
(2,-3)
\end{smallmatrix}$.
\end{center}

Finally, we conclude that the $g$-vectors for $\Delta_3$ are
\begin{center}
\begin{tikzpicture}[shorten >=1pt, auto, node distance=0cm,
   node_style/.style={font=},
   edge_style/.style={draw=black}]
\node[node_style] (v0) at (0,0) {$\begin{smallmatrix}
(1,0)\\
\oplus \\
(0,1)
\end{smallmatrix}$};
\node[node_style] (v1) at (12,0) {$\begin{smallmatrix}
(-1,0)\\
\oplus \\
(0,1)
\end{smallmatrix}$};
\node[node_style] (v) at (12,-6) {$\begin{smallmatrix}
(-1,0)\\
\oplus \\
(0,-1)
\end{smallmatrix}$};
\node[node_style] (v2) at (0,-2) {$\begin{smallmatrix}
(1,0)\\
\oplus \\
(3,-1)
\end{smallmatrix}$};
\node[node_style] (v21) at (0,-4) {$\begin{smallmatrix}
(2,-1)\\
\oplus \\
(3,-1)
\end{smallmatrix}$};
\node[node_style] (v212) at (0,-6) {$\begin{smallmatrix}
(2,-1)\\
\oplus \\
(3,-2)
\end{smallmatrix}$};
\node[node_style] (v2121) at (2,-6) {$\begin{smallmatrix}
(1,-1)\\
\oplus \\
(3,-2)
\end{smallmatrix}$};

\node[node_style] (v21212) at (4,-6) {$\begin{smallmatrix}
(1,-1)\\
\oplus \\
(2,-3)
\end{smallmatrix}$};
\node[node_style] (v212121) at (6,-6) {$\begin{smallmatrix}
(1,-2)\\
\oplus \\
(2,-3)
\end{smallmatrix}$};
\node[node_style] (v2121212) at (8,-6) {$\begin{smallmatrix}
(1,-2)\\
\oplus \\
(1,-3)
\end{smallmatrix}$};
\node[node_style] (v21212121) at (10,-6) {$\begin{smallmatrix}
(0,-1)\\
\oplus \\
(1,-3)
\end{smallmatrix}$};
\draw[->]  (v0) edge node{\ } (v1);
\draw[->]  (v1) edge node{\ } (v);
\draw[->]  (v0) edge node{\ } (v2);
\draw[->]  (v2) edge node{\ } (v21);
\draw[->]  (v21) edge node{\ } (v212);
\draw[->]  (v212) edge node{\ } (v2121);
\draw[->]  (v2121) edge node{\ } (v21212);
\draw[->]  (v21212) edge node{\ } (v212121);
\draw[->]  (v212121) edge node{\ } (v2121212);
\draw[->]  (v2121212) edge node{\ } (v21212121);
\draw[->]  (v21212121) edge node{\ } (v);
\end{tikzpicture}.
\end{center}
This implies that $\Delta_3$ is $\tau$-tilting finite.
\end{proof}

\begin{theorem}\label{result-Q2}
Let $\Lambda=KQ_2/I$ be a monomial algebra. Then, $\Lambda$ is $\tau$-tilting finite if and only if it does not have one of $\Omega_1$, $\Omega_1^{\emph{op}}$ and $\Omega_2$ as a quotient algebra.
\end{theorem}
\begin{proof}
If $\Lambda$ has one of $\Omega_1$, $\Omega_1^{\text{op}}$ and $\Omega_2$ as a quotient, then $\Lambda$ is obviously $\tau$-tilting infinite. If $\Lambda$ does not have one of $\Omega_1$, $\Omega_1^{\text{op}}$ and $\Omega_2$ as a quotient, we must have $\alpha^3\mu, \mu\beta^3, \alpha\mu\beta \in I$ such that $\alpha^3$ and $\beta^3$ are included in the center of $\Lambda$. By considering Lemma \ref{lem-center}, we deduce that either $\Lambda$ is a quotient of $\Delta_3$ or $\Lambda/\langle \alpha^3, \beta^3\rangle$ is a quotient of $\Delta_3$.
\end{proof}

\section{Extensions}

\subsection{Some cases with radical cube not zero}
\ \\
\vspace{-0.3cm}

We provide some piecemeal examples in this subsection, implying that our strategy is applicable in a more general setting.

\begin{proposition}\label{Omega-4}
Let $\Omega_4:=KQ_3/\langle\beta_1^3,\beta_2^2,\beta_2\beta_1,\mu\beta_2, \beta_1^2\beta_2\rangle$ with
\begin{center}
$Q_3: \xymatrix@C=1cm{1\ar[r]^{\mu}& 2 \ar@(ul,ur)^{\beta_1}\ar@(dl,dr)_{\beta_2}}$.
\end{center}
Then, $\Omega_4$ is minimal $\tau$-tilting infinite.
\end{proposition}
\begin{proof}
(1) We show that $\Omega_4$ is $\tau$-tilting infinite. Let $P_1$ and $P_2$ be the indecomposable projective $\Omega_4$-modules. Then,
\begin{center}
$P_1=$\scalebox{0.8}{$\vcenter{\xymatrix@C=0.1cm@R=0.2cm{
&e_1\ar@{-}[d]&\\
&\mu\ar@{-}[d]&\\
&\mu\beta_1\ar@{-}[dl]\ar@{-}[dr]&\\
\mu\beta_1^2 && \mu\beta_1\beta_2}}
\simeq \vcenter{\xymatrix@C=0.1cm@R=0.2cm{
&1\ar@{-}[d]&\\
&2\ar@{-}[d]&\\
&2\ar@{-}[dl]\ar@{-}[dr]&\\
2 && 2}}$} and $P_2=$\scalebox{0.8}{$\vcenter{\xymatrix@C=0.01cm@R=0.2cm{
&&e_2\ar@{-}[dl]\ar@{-}[dr]&\\
&\beta_1\ar@{-}[dl]\ar@{-}[dr]&&\beta_2\\
\beta_1^2&&\beta_1\beta_2&}}
\simeq \vcenter{\xymatrix@C=0.01cm@R=0.2cm{
&&2\ar@{-}[dl]\ar@{-}[dr]&\\
&2\ar@{-}[dl]\ar@{-}[dr]&&2\\
2&&2&}}$}.
\end{center}
Obviously,
\begin{center}
$\begin{aligned}
\mathsf{Hom}_{\Omega_4}(P_1,P_1)&=1,\ \mathsf{Hom}_{\Omega_4}(P_1,P_2)=0,\\
\mathsf{Hom}_{\Omega_4}(P_2,P_2)&=\mathsf{span}\{1, \beta_1, \beta_2, \beta_1^2, \beta_1\beta_2\},\\
\mathsf{Hom}_{\Omega_4}(P_2,P_1)&=\mathsf{span}\{\mu, \mu\beta_1, \mu\beta_1^2, \mu\beta_1\beta_2\}.
\end{aligned}$
\end{center}

We start with the following tilting complex,
\begin{center}
$\mu_{2}^-(\Omega_4)=\left [\begin{smallmatrix}
0\longrightarrow P_1\\
\oplus \\
P_2\overset{f}{\longrightarrow} P_1^{\oplus 4}
\end{smallmatrix}  \right ]$, where $f=\left ( \begin{smallmatrix}
\mu\\
\mu\beta_1\\
\mu\beta_1^2\\
\mu\beta_1\beta_2
\end{smallmatrix} \right )$.
\end{center}
Let $X:=(0\longrightarrow P_1)$ and $Y:=(P_2\overset{f}{\longrightarrow} P_1^{\oplus 4})$. Then,
\begin{center}
$\mathsf{Hom}_{\mathcal{K}_{\Omega_4}}(Y,X)=0$, $\mathsf{Hom}_{\mathcal{K}_{\Omega_4}}(X,X)=\{(0,1)\}$,
\end{center}

\begin{center}
$\mathsf{Hom}_{\mathcal{K}_{\Omega_4}}(X,Y)=\mathsf{span} \left \{\left (0, \left ( \begin{smallmatrix}
1\\
0\\
0\\
0
\end{smallmatrix} \right )\right ), \left (0, \left ( \begin{smallmatrix}
0\\
1\\
0\\
0
\end{smallmatrix} \right )\right ), \left (0, \left ( \begin{smallmatrix}
0\\
0\\
1\\
0
\end{smallmatrix} \right )\right ), \left (0, \left ( \begin{smallmatrix}
0\\
0\\
0\\
1
\end{smallmatrix} \right )\right )\right \}$,
\end{center}

\begin{center}
$\begin{aligned}\mathsf{Hom}_{\mathcal{K}_{\Omega_4}}(Y,Y)=\mathsf{span}
&\left\{\begin{matrix}
\left (1, \left ( \begin{smallmatrix}
1&0&0&0\\
0&1&0&0\\
0&0&1&0\\
0&0&0&1
\end{smallmatrix} \right )\right ), \left (\beta_1, \left ( \begin{smallmatrix}
0&1&0&0\\
0&0&1&0\\
0&0&0&0\\
0&0&0&0
\end{smallmatrix} \right )\right ), \left (\beta_2, \left ( \begin{smallmatrix}
0&0&0&0\\
0&0&0&1\\
0&0&0&0\\
0&0&0&0
\end{smallmatrix} \right )\right ),
\end{matrix}\right. \\
&\ \ \left.\begin{matrix}
\left (\beta_1^2, \left ( \begin{smallmatrix}
0&0&1&0\\
0&0&0&0\\
0&0&0&0\\
0&0&0&0
\end{smallmatrix} \right )\right ), \left (\beta_1\beta_2, \left ( \begin{smallmatrix}
0&0&0&1\\
0&0&0&0\\
0&0&0&0\\
0&0&0&0
\end{smallmatrix} \right )\right )
\end{matrix}\right\}.
\end{aligned}$
\end{center}
Set
\begin{center}
$b_1:=\left (\beta_1, \left ( \begin{smallmatrix}
0&1&0&0\\
0&0&1&0\\
0&0&0&0\\
0&0&0&0
\end{smallmatrix} \right )\right )$, $b_2:=\left (\beta_2, \left ( \begin{smallmatrix}
0&0&0&0\\
0&0&0&1\\
0&0&0&0\\
0&0&0&0
\end{smallmatrix} \right )\right )$, $c_1:=\left (0, \left ( \begin{smallmatrix}
0\\
0\\
1\\
0
\end{smallmatrix} \right )\right )$, $c_2:=\left (0, \left ( \begin{smallmatrix}
0\\
0\\
0\\
1
\end{smallmatrix} \right )\right )$,
\end{center}
and then,
\begin{center}
$b_1^3=b_2^2=b_2b_1=b_1c_2=b_2c_1=b_1^2b_2=0$,
\end{center}

\begin{center}
$b_1c_1=b_2c_2=\left (0, \left ( \begin{smallmatrix}
0\\
1\\
0\\
0
\end{smallmatrix} \right )\right )$, $b_1^2c_1=b_1b_2c_2=\left (0, \left ( \begin{smallmatrix}
1\\
0\\
0\\
0
\end{smallmatrix} \right )\right )$.
\end{center}
It turns out that $\mathsf{End}_{\mathcal{K}_{\Omega_4}}(\mu_{2}^-(\Omega_4))$ is isomorphic to $\Delta_4:=KQ/I$ with
\begin{center}
$Q: \xymatrix@C=1cm{1 & 2\ar@<0.5ex>[l]^{c_1}\ar@<-0.5ex>[l]_{c_2} \ar@(dl,dr)_{b_2}\ar@(ul,ur)^{b_1}}$ and
$I=\langle b_1^3,b_2^2,b_2b_1,b_1c_2,b_2c_1, b_1^2b_2, b_1c_1-b_2c_2\rangle$.
\end{center}
Thus, $\Omega_4$ is $\tau$-tilting infinite following Lemma \ref{lemma-derived-infinite}.

(2) For the minimality, we should first consider
\begin{center}
$\widetilde{\Omega}_4:=\Omega_4/\langle x\mu\beta_1\beta_2+y\mu\beta_1^2\rangle$,
\end{center}
since the socle of $\Omega_4$ is $K\mu\beta_1^2\oplus K\mu\beta_1\beta_2\oplus K\beta_1^2\oplus K\beta_1\beta_2$. However, if both $x$ and $y$ are not zero, we may get an algebra isomorphism by replacing $\beta_1$ with $\beta_2-(y/x)\beta_1$. Thus, it suffices to consider
\begin{center}
$\overline{\Omega}_4:=\Omega_4/\langle \mu\beta_1\beta_2\rangle$ and ${\overline{\Omega}_{4}}':=\Omega_4/\langle \mu\beta_1^2\rangle$,
\end{center}

By using Lemma \ref{lem-center}, it is not difficult to find that $\overline{\Omega}_4$ shares the same $g$-vectors with $\overline{\Omega}_1$ which is known to be $\tau$-tilting finite by \cite[Lemma 3.3]{W-two-point}. By direct calculation, the $g$-vectors for ${\overline{\Omega}_4}'$ are given by
\begin{center}
\begin{tikzpicture}[shorten >=1pt, auto, node distance=0cm,
   node_style/.style={font=},
   edge_style/.style={draw=black}]
\node[node_style] (v) at (-2,0) {$\begin{smallmatrix}
(1,0)\\
\oplus \\
(0,1)
\end{smallmatrix}$};
\node[node_style] (v3) at (0,0) {$\begin{smallmatrix}
(1,0)\\
\oplus \\
(3,-1)
\end{smallmatrix}$};
\node[node_style] (v1) at (-2,-2) {$\begin{smallmatrix}
(-1,0)\\
\oplus \\
(0,1)
\end{smallmatrix}$};
\node[node_style] (v2) at (2,0) {$\begin{smallmatrix}
(3,-1)\\
\oplus \\
(2,-1)
\end{smallmatrix}$};
\node[node_style] (v21) at (4,0) {$\begin{smallmatrix}
(1,-1)\\
\oplus \\
(2,-1)
\end{smallmatrix}$};
\node[node_style] (v212) at (6,0) {$\begin{smallmatrix}
(1,-1)\\
\oplus \\
(0,-1)
\end{smallmatrix}$};
\node[node_style] (v0) at (6,-2) {$\begin{smallmatrix}
(-1,0)\\
\oplus \\
(0,-1)
\end{smallmatrix}$.};

\draw[->]  (v) edge node{ } (v1);
\draw[->]  (v1) edge node{ } (v0);
\draw[->]  (v) edge node{ } (v3);
\draw[->]  (v3) edge node{ } (v2);
\draw[->]  (v2) edge node{ } (v21);
\draw[->]  (v21) edge node{ } (v212);
\draw[->]  (v212) edge node{ } (v0);
\end{tikzpicture}
\end{center}
This implies that $\Omega_4$ is minimal $\tau$-tilting infinite.
\end{proof}

\begin{corollary}\label{Omega-5}
Let $\Omega_5:=KQ/\langle \beta^3, \beta\nu, \nu\mu\nu, \nu\mu\beta^2 \rangle$ with
\begin{center}
$Q: \xymatrix@C=1cm{1\ar@<0.5ex>[r]^{\mu}&2\ar@<0.5ex>[l]^{\nu}\ar@(ur,dr)^{\beta}}$.
\end{center}
Then, $\Omega_5$ is minimal $\tau$-tilting infinite.
\end{corollary}
\begin{proof}
By Lemma \ref{lem-oppsite-algebra} and Proposition \ref{Omega-4}, $\Omega_4^\text{op}$ is $\tau$-tilting infinite. We show that   $\Omega_5$ is isomorphic to $\mathsf{End}_{\mathcal{K}_{\Omega_4^\text{op}}}(\mu_{1}^-(\Omega_4^\text{op}))$, and then, $\Omega_5$ is $\tau$-tilting infinite by Lemma \ref{lemma-derived-infinite}.

In the quiver of $\Omega_4^\text{op}$, we use $\mu, \beta_1, \beta_2$ to indicate $\mu^\ast, \beta_1^\ast, \beta_2^\ast$ for simplicity. The indecomposable projective $\Omega_4^\text{op}$-modules are given by
\begin{center}
$P_1=e_1\simeq 1$ and $P_2=$\scalebox{0.8}{$\vcenter{\xymatrix@C=0.2cm@R=0.3cm{
&e_2\ar@{-}[dl]\ar@{-}[d]\ar@{-}[dr]&\\
\beta_1\ar@{-}[d]\ar@{-}[dr]&\mu&\beta_2\ar@{-}[d]\\
\beta_1^2\ar@{-}[d]&\beta_1\mu&\beta_2\beta_1\ar@{-}[d]\\
\beta_1^2\mu&&\beta_2\beta_1\mu}}
\simeq \vcenter{\xymatrix@C=0.5cm@R=0.3cm{
&2\ar@{-}[dl]\ar@{-}[d]\ar@{-}[dr]&\\
2\ar@{-}[d]\ar@{-}[dr]&1&2\ar@{-}[d]\\
2\ar@{-}[d]&1&2\ar@{-}[d]\\
1&&1}}$}.
\end{center}
Then,
\begin{center}
$\mu_{1}^-(\Omega_4^\text{op})=\left [\begin{smallmatrix}
P_1\overset{\mu}{\longrightarrow} P_2\\
\oplus \\
0\longrightarrow P_2
\end{smallmatrix}  \right ]$.
\end{center}
Let $X:=(P_1\overset{\mu}{\longrightarrow} P_2)$ and $Y:=(0\longrightarrow P_2)$. Then,
\begin{center}
$\mathsf{Hom}_{\mathcal{K}_{\Omega_4^{\text{op}}}}(X,Y)=\mathsf{span}\{(0,\beta_2)\}$, $\mathsf{Hom}_{\mathcal{K}_{\Omega_4^{\text{op}}}}(X,X)=\mathsf{span}\{(1,1), (0,\beta_2)\}$,
\end{center}

\begin{center}
$\mathsf{Hom}_{\mathcal{K}_{\Omega_4^{\text{op}}}}(Y,Y)=\mathsf{span} \left \{(0,1), (0, \beta_1),(0, \beta_2),(0, \beta_1^2), (0, \beta_2\beta_1)\right \}$,
\end{center}

\begin{center}
$\mathsf{Hom}_{\mathcal{K}_{\Omega_4^{\text{op}}}}(Y,X)=\mathsf{span} \left \{(0,1), (0, \beta_1),(0, \beta_2),(0, \beta_1^2), (0, \beta_2\beta_1)\right \}$.
\end{center}
Set $m:=(0,1), b:=(0,\beta_1), n:=(0,\beta_2)$. We have
\begin{center}
$b^3=bn=nmn=nmb^2=0$, $mb=(0,\beta_1)$, $nm=mn=mnm=(0,\beta_2)$,

$b^2=mb^2=(0, \beta_1^2)$, $nmb=mnmb=(0,\beta_2\beta_1)$.
\end{center}
It is not difficult to see that $\mathsf{End}_{\mathcal{K}_{\Omega_4^\text{op}}}(\mu_{1}^-(\Omega_4^\text{op}))\simeq \Omega_5$.

We consider the quotient of $\Omega_5$ modulo the ideal generated by $\mu\nu\mu\beta$, since the socle of $\Omega_5$ is $K\mu\nu\mu\beta\oplus K\nu\mu\beta$. By Lemma \ref{lem-center},
\begin{center}
$\mathsf{2\text{-}silt}\ (\Omega_5/\langle \mu\nu\mu\beta \rangle) \simeq \mathsf{2\text{-}silt}\ (\Omega_5/\langle \nu\mu\beta \rangle)\simeq \mathsf{2\text{-}silt}\ (\Omega_5/\langle \nu\mu+\mu\nu \rangle)$.
\end{center}
Define $\overline{\Omega}_5:=\Omega_5/\langle \nu\mu+\mu\nu \rangle$. By a similar calculation to $\overline{\Omega}_1$ (see  \cite[Lemma 3.3]{W-two-point}), the $g$-vectors for $\overline{\Omega}_5$ are
\begin{center}
\begin{tikzpicture}[shorten >=1pt, auto, node distance=0cm,
   node_style/.style={font=},
   edge_style/.style={draw=black}]
\node[node_style] (v) at (-2,0) {$\begin{smallmatrix}
(1,0)\\
\oplus \\
(0,1)
\end{smallmatrix}$};
\node[node_style] (v3) at (0,0) {$\begin{smallmatrix}
(1,0)\\
\oplus \\
(3,-1)
\end{smallmatrix}$};
\node[node_style] (v1) at (-2,-2) {$\begin{smallmatrix}
(-1,1)\\
\oplus \\
(0,1)
\end{smallmatrix}$};
\node[node_style] (v12) at (2,-2) {$\begin{smallmatrix}
(-1,1)\\
\oplus \\
(-1,0)
\end{smallmatrix}$};
\node[node_style] (v2) at (2,0) {$\begin{smallmatrix}
(2,-1)\\
\oplus \\
(3,-1)
\end{smallmatrix}$};
\node[node_style] (v21) at (4,0) {$\begin{smallmatrix}
(2,-1)\\
\oplus \\
(3,-2)
\end{smallmatrix}$};
\node[node_style] (v212) at (6,0) {$\begin{smallmatrix}
(1,-1)\\
\oplus \\
(3,-2)
\end{smallmatrix}$};
\node[node_style] (v2121) at (8,0) {$\begin{smallmatrix}
(1,-1)\\
\oplus \\
(0,-1)
\end{smallmatrix}$};
\node[node_style] (v0) at (8,-2) {$\begin{smallmatrix}
(-1,0)\\
\oplus \\
(0,-1)
\end{smallmatrix}$.};

\draw[->]  (v) edge node{ } (v1);
\draw[->]  (v1) edge node{ } (v12);
\draw[->]  (v12) edge node{ } (v0);
\draw[->]  (v) edge node{ } (v3);
\draw[->]  (v3) edge node{ } (v2);
\draw[->]  (v2) edge node{ } (v21);
\draw[->]  (v21) edge node{ } (v212);
\draw[->]  (v212) edge node{ } (v2121);
\draw[->]  (v2121) edge node{ } (v0);
\end{tikzpicture}
\end{center}
This implies that $\overline{\Omega}_5$ is $\tau$-tilting finite and $\Omega_5$ is minimal $\tau$-tilting infinite.
\end{proof}

\begin{remark}
Let $\Lambda=KQ_3/I$ be a monomial algebra associated with quiver $Q_3$ and $P_1$ the indecomposable projective $\Lambda$-module at vertex 1. Then, $\Lambda$ is $\tau$-tilting infinite if $P_1$ is not uniserial, i.e., $P_1$ is of form
\begin{center}
\scalebox{0.8}{$\vcenter{\xymatrix@C=0.1cm@R=0.1cm{
&1\ar@{-}[d]&\\
&2\ar@{-}[dl]\ar@{-}[dr]&\\
2 && 2}}$, \ $\vcenter{\xymatrix@C=0.1cm@R=0.1cm{
&1\ar@{-}[d]&\\
&2\ar@{-}[d]&\\
&2\ar@{-}[dl]\ar@{-}[dr]&\\
2 && 2}}$, \ $\vcenter{\xymatrix@C=0.1cm@R=0.1cm{
&1\ar@{-}[d]&\\
&2\ar@{-}[d]&\\
&2\ar@{-}[d]&\\
&2\ar@{-}[dl]\ar@{-}[dr]&\\
2 && 2}}$, \ $\cdots$, $\vcenter{\xymatrix@C=0.1cm@R=0.1cm{
&1\ar@{-}[d]&\\
&2\ar@{-}[d]&\\
&\vdots &\\
&2\ar@{-}[dl]\ar@{-}[dr]&\\
2 && 2}}$, \ $\vcenter{\xymatrix@C=0.1cm@R=0.1cm{
&&1\ar@{-}[d]&\\
&&2\ar@{-}[dl]\ar@{-}[dr]&\\
&2\ar@{-}[dl]\ar@{-}[dr] && 2\\
2&&2&}}$, \ $\cdots$.
}
\end{center}
In Proposition \ref{Omega-3} and Proposition \ref{Omega-4}, we established the statement for the first two examples; others can be proved using a similar approach (but more involved). Moreover, if $\Lambda$ is $\tau$-tilting infinite, then the endomorphism algebra $\mathsf{End}_{\mathcal{K}_{\Lambda^\text{op}}}(\mu_{1}^-(\Lambda^\text{op}))$ gives a new $\tau$-tilting infinite algebra by a similar calculation with Corollary \ref{Omega-5}, but it is not necessarily minimal in general.
\end{remark}

\begin{theorem}\label{result-Q(1111)}
Let $\Lambda=KQ/I$ be a monomial algebra with $\mathsf{rad}^5\ \Lambda=0$ and
\begin{center}
$Q: \xymatrix@C=1cm{1\ar@(dl,ul)^{\alpha}\ar@<0.5ex>[r]^{\mu}&2\ar@<0.5ex>[l]^{\nu}\ar@(ur,dr)^{\beta}}$.
\end{center}
Then, $\Lambda$ is $\tau$-tilting finite if and only if it does not have one of $\Omega_1$, $\Omega_2$, $\Omega_5$ and their opposite algebras as a quotient algebra.
\end{theorem}
\begin{proof}
We assume that $\Lambda$ does not have one of $\Omega_1$, $\Omega_2$, $\Omega_5$ and their opposite algebras as a quotient algebra. This is equivalent to saying that
\begin{center}
$\alpha^3\mu, \mu\beta^3, \beta^3\nu, \nu\alpha^3, \alpha\mu\beta, \beta\nu\alpha, \mu\nu\mu\beta, \beta\nu\mu\nu, \nu\mu\nu\alpha,\alpha\mu\nu\mu \in I$.
\end{center}
Taking the condition $\mathsf{rad}^5\ \Lambda=0$ and Lemma \ref{lem-center} into our consideration, we get
\begin{center}
$\begin{aligned}\mathsf{2\text{-}silt}\ \Lambda
&\simeq \mathsf{2\text{-}silt}\ \left(\Lambda/ \left \langle
\begin{matrix}
\alpha^3,\mu\nu\mu\nu, \nu\alpha^2\mu, \alpha^2\mu\nu, \alpha\mu\nu\alpha, \mu\nu\alpha^2\\
\beta^3, \nu\mu\nu\mu, \mu\beta^2\nu, \beta^2\nu\mu, \beta\nu\mu\beta, \nu\mu\beta^2
\end{matrix}\right \rangle\right)\\
&\simeq \mathsf{2\text{-}silt}\ \left(\Lambda/ \left \langle
\begin{matrix}
\alpha^3,\mu\nu\mu\nu, \nu\alpha^2\mu, \alpha^2\mu\nu, \alpha\mu\nu\alpha, \mu\nu\alpha^2\\
\beta^3, \nu\mu\nu\mu, \mu\beta^2\nu, \beta^2\nu\mu, \beta\nu\mu\beta, \nu\mu\beta^2\\
\alpha\mu\nu+\mu\nu\alpha+\nu\alpha\mu, \beta\nu\mu+\nu\mu\beta+\mu\beta\nu
\end{matrix}\right \rangle\right)\\
&\simeq \mathsf{2\text{-}silt}\ \left(\Lambda/ \left \langle
\begin{matrix}
\alpha^3,\mu\nu\mu\nu, \nu\alpha^2\mu, \mu\nu\alpha^2, \beta^3, \nu\mu\nu\mu, \mu\beta^2\nu, \beta^2\nu\mu\\
\mu\nu\alpha+\nu\alpha\mu, \beta\nu\mu+\mu\beta\nu, \alpha\mu\nu, \nu\mu\beta
\end{matrix}\right \rangle\right)\\
&\simeq \mathsf{2\text{-}silt}\ \left(\Lambda/ \left \langle
\alpha^3, \beta^3, \nu\alpha\mu, \mu\beta\nu, \nu\alpha^2\mu, \mu\beta^2\nu, \mu\nu+\nu\mu
\right \rangle\right)
\end{aligned}$.
\end{center}
Set
\begin{center}
$\Delta_5:=KQ/\left \langle
\alpha^3, \beta^3, \mu\nu+\nu\mu, \alpha\mu\beta, \beta\nu\alpha, \nu\alpha\mu, \mu\beta\nu, \nu\alpha^2\mu, \mu\beta^2\nu
\right \rangle$.
\end{center}
It turns out either $\Lambda$ is a quotient of $\Delta_5$ or $\mathsf{2\text{-}silt}\ \Lambda\simeq \mathsf{2\text{-}silt}\ \Delta_5$.

We show that $\Delta_5$ is $\tau$-tilting finite. The indecomposable projective $\Delta_5$-modules are
\begin{center}
$P_1=$\scalebox{0.8}{$\vcenter{\xymatrix@C=0.1cm@R=0.2cm{
&e_1\ar@{-}[dl]\ar@{-}[dr]&&&\\
\mu\ar@{-}[d]&&\alpha\ar@{-}[dl]\ar@{-}[dr]&&\\
\mu\beta\ar@{-}[d]&\alpha\mu&&\alpha^2\ar@{-}[dl]&\\
\mu\beta^2&&\alpha^2\mu&& }}$} $\simeq$ \scalebox{0.7}{$\vcenter{\xymatrix@C=0.2cm@R=0.2cm{
&1\ar@{-}[dl]\ar@{-}[dr]&&&\\
2\ar@{-}[d]&&1\ar@{-}[dl]\ar@{-}[dr]&&\\
2\ar@{-}[d]&2&&1\ar@{-}[dl]&\\
2&&2&& }}$} and $P_2=$\scalebox{0.8}{$\vcenter{\xymatrix@C=0.1cm@R=0.2cm{
&e_2\ar@{-}[dl]\ar@{-}[dr]&&&\\
\nu\ar@{-}[d]&&\beta\ar@{-}[dl]\ar@{-}[dr]&&\\
\nu\alpha\ar@{-}[d]&\beta\nu&&\beta^2\ar@{-}[dl]&\\
\nu\alpha^2&&\beta^2\nu&& }}$} $\simeq$ \scalebox{0.7}{$\vcenter{\xymatrix@C=0.2cm@R=0.2cm{
&2\ar@{-}[dl]\ar@{-}[dr]&&&\\
1\ar@{-}[d]&&2\ar@{-}[dl]\ar@{-}[dr]&&\\
1\ar@{-}[d]&1&&2\ar@{-}[dl]&\\
1&&1&& }}$}.
\end{center}
There exist two algebra isomorphisms $\sigma, \sigma':{\Delta}_5^\text{op}\rightarrow \Delta_5$ satisfying
\begin{center}
$\sigma(e_1^\ast)=e_2$, $\sigma(e_2^\ast)=e_1$ and $\sigma'(e_1^\ast)=e_1$, $\sigma'(e_2^\ast)=e_2$.
\end{center}
By similar analysis in the proof of Theorem \ref{result-rad-cube-zero} and Proposition \ref{Delta-3}, we infer that the $g$-vectors for $\Delta_5$ are
\begin{center}
\begin{tikzpicture}[shorten >=1pt, auto, node distance=0cm,
   node_style/.style={font=},
   edge_style/.style={draw=black}]
\node[node_style] (v0) at (0,0) {$\begin{smallmatrix}
(1,0)\\
\oplus \\
(0,1)
\end{smallmatrix}$};
\node[node_style] (v1) at (2,0) {$\begin{smallmatrix}
(-1,3)\\
\oplus \\
(0,1)
\end{smallmatrix}$};
\node[node_style] (v12) at (4,0) {$\begin{smallmatrix}
(-1,3)\\
\oplus \\
(-1,2)
\end{smallmatrix}$};
\node[node_style] (v121) at (6,0) {$\begin{smallmatrix}
(-2,3)\\
\oplus \\
(-1,2)
\end{smallmatrix}$};
\node[node_style] (v1212) at (8,0) {$\begin{smallmatrix}
(-2,3)\\
\oplus \\
(-1,1)
\end{smallmatrix}$};
\node[node_style] (v12121) at (10,0) {$\begin{smallmatrix}
(-3,2)\\
\oplus \\
(-1,1)
\end{smallmatrix}$};
\node[node_style] (v121212) at (12,0) {$\begin{smallmatrix}
(-3,2)\\
\oplus \\
(-2,1)
\end{smallmatrix}$};
\node[node_style] (v1212121) at (12,-2) {$\begin{smallmatrix}
(-3,1)\\
\oplus \\
(-2,1)
\end{smallmatrix}$};
\node[node_style] (v12121212) at (12,-4) {$\begin{smallmatrix}
(-3,1)\\
\oplus \\
(-1,0)
\end{smallmatrix}$};
\node[node_style] (v) at (12,-6) {$\begin{smallmatrix}
(-1,0)\\
\oplus \\
(0,-1)
\end{smallmatrix}$};
\node[node_style] (v2) at (0,-2) {$\begin{smallmatrix}
(1,0)\\
\oplus \\
(3,-1)
\end{smallmatrix}$};
\node[node_style] (v21) at (0,-4) {$\begin{smallmatrix}
(2,-1)\\
\oplus \\
(3,-1)
\end{smallmatrix}$};
\node[node_style] (v212) at (0,-6) {$\begin{smallmatrix}
(2,-1)\\
\oplus \\
(3,-2)
\end{smallmatrix}$};
\node[node_style] (v2121) at (2,-6) {$\begin{smallmatrix}
(1,-1)\\
\oplus \\
(3,-2)
\end{smallmatrix}$};

\node[node_style] (v21212) at (4,-6) {$\begin{smallmatrix}
(1,-1)\\
\oplus \\
(2,-3)
\end{smallmatrix}$};
\node[node_style] (v212121) at (6,-6) {$\begin{smallmatrix}
(1,-2)\\
\oplus \\
(2,-3)
\end{smallmatrix}$};
\node[node_style] (v2121212) at (8,-6) {$\begin{smallmatrix}
(1,-2)\\
\oplus \\
(1,-3)
\end{smallmatrix}$};
\node[node_style] (v21212121) at (10,-6) {$\begin{smallmatrix}
(0,-1)\\
\oplus \\
(1,-3)
\end{smallmatrix}$};
\draw[->]  (v0) edge node{\ } (v1);
\draw[->]  (v1) edge node{\ } (v12);
\draw[->]  (v12) edge node{\ } (v121);
\draw[->]  (v121) edge node{\ } (v1212);
\draw[->]  (v1212) edge node{\ } (v12121);
\draw[->]  (v12121) edge node{\ } (v121212);
\draw[->]  (v121212) edge node{\ } (v1212121);
\draw[->]  (v1212121) edge node{\ } (v12121212);
\draw[->]  (v12121212) edge node{\ } (v);
\draw[->]  (v0) edge node{\ } (v2);
\draw[->]  (v2) edge node{\ } (v21);
\draw[->]  (v21) edge node{\ } (v212);
\draw[->]  (v212) edge node{\ } (v2121);
\draw[->]  (v2121) edge node{\ } (v21212);
\draw[->]  (v21212) edge node{\ } (v212121);
\draw[->]  (v212121) edge node{\ } (v2121212);
\draw[->]  (v2121212) edge node{\ } (v21212121);
\draw[->]  (v21212121) edge node{\ } (v);
\end{tikzpicture}.
\end{center}
We then deduce that $\Lambda$ is also $\tau$-tilting finite.
\end{proof}

\subsection{Derived equivalence of the Kronecker algebra}
\ \\
\vspace{-0.3cm}

\begin{proposition}\label{result-derived-eq-kronecker}
If $\Lambda$ is derived equivalent to the Kronecker algebra $\Delta_1=K(\xymatrix@C=0.7cm{1\ar@<0.5ex>[r]^{ }\ar@<-0.5ex>[r]_{ }&2})$, then $\Lambda$ is isomorphic to $\Delta_1$ or $\Delta_1^{\emph{op}}$.
\end{proposition}
\begin{proof}
It is well-known from \cite{Rickard-tilting-complex} that $\Lambda$ is derived equivalent to $\Delta_1$ if and only if there exists a tilting complex $T\in \mathcal{K}_{\Delta_1}$ such that $\Lambda \simeq\mathsf{End}_{\mathcal{K}_{\Delta_1}}\ (T)$. In the following, we first list all tilting complexes in $\mathcal{K}_{\Delta_1}$ and then, check the corresponding endomorphism algebras.

We start with the two-term silting complexes in $\mathcal{K}_{\Delta_1}$. Let $P_1$ and $P_2$ be the indecomposable projective $\Delta_1$-modules. Set $X_0:=(0\longrightarrow P_2)$ and $X_1:=(0\longrightarrow P_1)$. Then, the Hasse quiver $\mathcal{H}(\mathsf{2\text{-}silt}\ \Delta_1)$ is displayed by
\begin{center}
$\vcenter{\xymatrix@C=0.6cm{X_0\oplus X_1\ar[d]\ar[rrrrr]&&&&&X_0\oplus X_1[1]\ar[d]\\
X_1\oplus X_2\ar[r]&X_2\oplus X_3\ar[r]&\cdots\ar[r]& X_{-2}\oplus X_{-1}\ar[r]&X_{-1}\oplus X_0[1]\ar[r]&(X_0\oplus X_1)[1]}}$,
\end{center}
where the corresponding $g$-vectors are given in Remark \ref{Hasse-quiver-kronecker-alg}. It is shown in \cite[Theorem 3.1]{AI-silting} that $\Delta_1$ is silting connected, that is, any silting complex in
$\mathcal{K}_{\Delta_1}$ could be obtained from $X_0\oplus X_1$ by iterated irreducible (left or right) silting mutations. It then follows from \cite[Example 2.46]{AI-silting} that the silting quiver of $\Delta_1$ is displayed in Figure 1, in which $\xymatrix@C=1cm{X\ar[r]|-{[1]}&Y}$ means $\xymatrix@C=0.7cm{X\ar[r]&Y[1]}$.
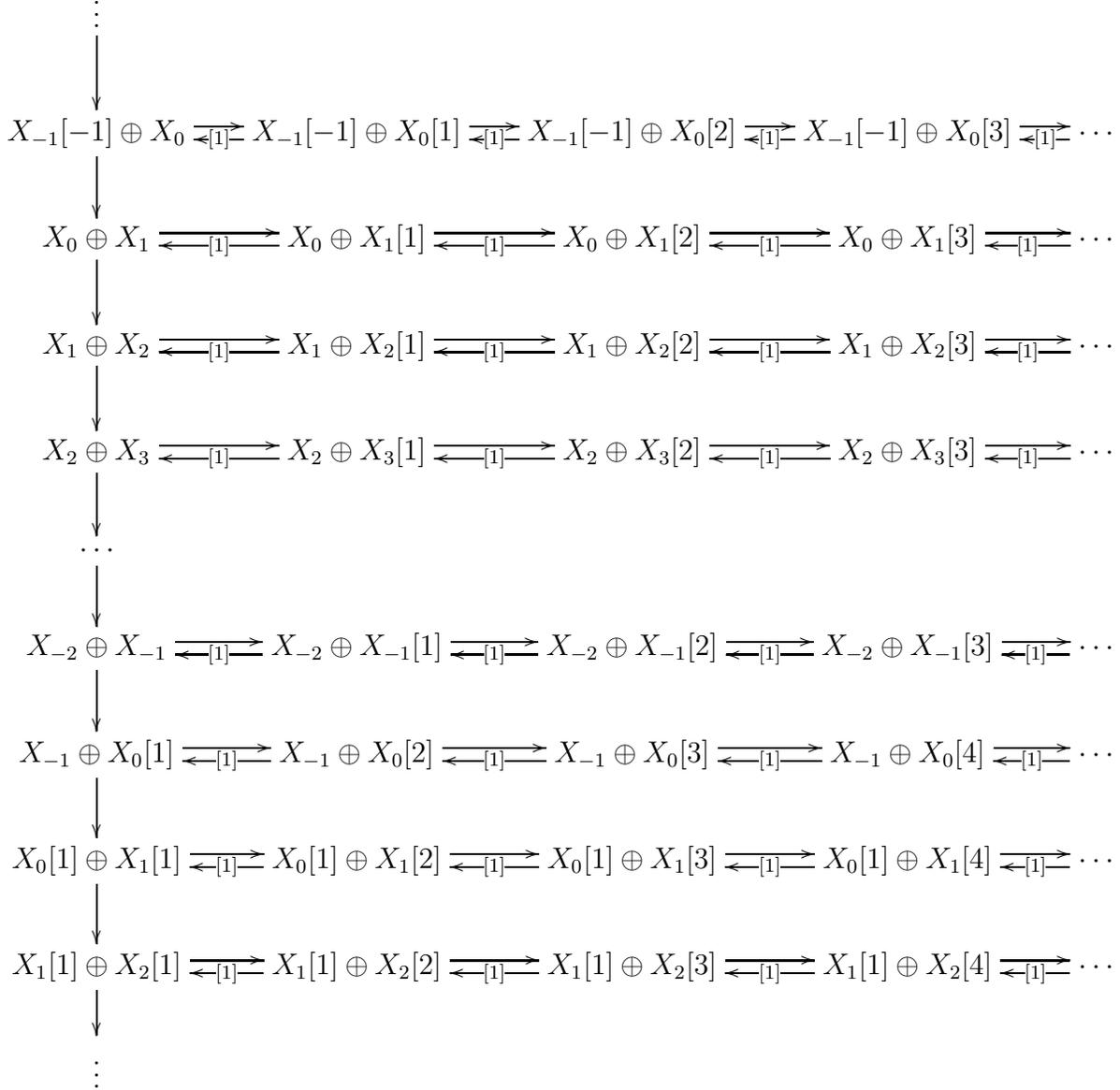
\begin{figure}
\label{Fiture}\caption{The silting quiver of the Kronecker algebra $\Delta_1$}
$\xymatrix@C=0.7cm{
\vdots\ar[d]&\\
X_{-1}[-1]\oplus X_0\ar[d]\ar@<0.5ex>[r]&X_{-1}[-1]\oplus X_0[1]\ar@<0.5ex>[r]\ar@<0.5ex>[l]|-{[1]}&X_{-1}[-1]\oplus X_0[2]\ar@<0.5ex>[r]\ar@<0.5ex>[l]|-{[1]}&X_{-1}[-1]\oplus X_0[3]\ar@<0.5ex>[r]\ar@<0.5ex>[l]|-{[1]}&\cdots\ar@<0.5ex>[l]|-{[1]}\\
X_0\oplus X_1\ar[d]\ar@<0.5ex>[r]&X_0\oplus X_1[1]\ar@<0.5ex>[r]\ar@<0.5ex>[l]|-{[1]}&X_0\oplus X_1[2]\ar@<0.5ex>[r]\ar@<0.5ex>[l]|-{[1]}&X_0\oplus X_1[3]\ar@<0.5ex>[r]\ar@<0.5ex>[l]|-{[1]}&\cdots\ar@<0.5ex>[l]|-{[1]}\\
X_1\oplus X_2\ar[d]\ar@<0.5ex>[r]&X_1\oplus X_2[1]\ar@<0.5ex>[r]\ar@<0.5ex>[l]|-{[1]}&X_1\oplus X_2[2]\ar@<0.5ex>[r]\ar@<0.5ex>[l]|-{[1]}&X_1\oplus X_2[3]\ar@<0.5ex>[r]\ar@<0.5ex>[l]|-{[1]}&\cdots\ar@<0.5ex>[l]|-{[1]}\\
X_2\oplus X_3\ar[d]\ar@<0.5ex>[r]&X_2\oplus X_3[1]\ar@<0.5ex>[r]\ar@<0.5ex>[l]|-{[1]}&X_2\oplus X_3[2]\ar@<0.5ex>[r]\ar@<0.5ex>[l]|-{[1]}&X_2\oplus X_3[3]\ar@<0.5ex>[r]\ar@<0.5ex>[l]|-{[1]}&\cdots\ar@<0.5ex>[l]|-{[1]}\\
\cdots\ar[d]&\\
X_{-2}\oplus X_{-1}\ar[d]\ar@<0.5ex>[r]&X_{-2}\oplus X_{-1}[1]\ar@<0.5ex>[r]\ar@<0.5ex>[l]|-{[1]}&X_{-2}\oplus X_{-1}[2]\ar@<0.5ex>[r]\ar@<0.5ex>[l]|-{[1]}&X_{-2}\oplus X_{-1}[3]\ar@<0.5ex>[r]\ar@<0.5ex>[l]|-{[1]}&\cdots\ar@<0.5ex>[l]|-{[1]}\\
X_{-1}\oplus X_0[1]\ar[d]\ar@<0.5ex>[r]&X_{-1}\oplus X_0[2]\ar@<0.5ex>[r]\ar@<0.5ex>[l]|-{[1]}&X_{-1}\oplus X_0[3]\ar@<0.5ex>[r]\ar@<0.5ex>[l]|-{[1]}&X_{-1}\oplus X_0[4]\ar@<0.5ex>[r]\ar@<0.5ex>[l]|-{[1]}&\cdots\ar@<0.5ex>[l]|-{[1]}\\
X_0[1]\oplus X_1[1]\ar[d]\ar@<0.5ex>[r]&X_0[1]\oplus X_1[2]\ar@<0.5ex>[r]\ar@<0.5ex>[l]|-{[1]}&X_0[1]\oplus X_1[3]\ar@<0.5ex>[r]\ar@<0.5ex>[l]|-{[1]}&X_0[1]\oplus X_1[4]\ar@<0.5ex>[r]\ar@<0.5ex>[l]|-{[1]}&\cdots\ar@<0.5ex>[l]|-{[1]}\\
X_1[1]\oplus X_2[1]\ar[d]\ar@<0.5ex>[r]&X_1[1]\oplus X_2[2]\ar@<0.5ex>[r]\ar@<0.5ex>[l]|-{[1]}&X_1[1]\oplus X_2[3]\ar@<0.5ex>[r]\ar@<0.5ex>[l]|-{[1]}&X_1[1]\oplus X_2[4]\ar@<0.5ex>[r]\ar@<0.5ex>[l]|-{[1]}&\cdots\ar@<0.5ex>[l]|-{[1]}\\
\vdots&}$
\end{figure}

We define $\mathcal{X}:=\{X_0\oplus X_1, X_1\oplus X_2, X_2\oplus X_3, \cdots, X_{-1}\oplus X_0[1],(X_0\oplus X_1)[1]\}$ and
\begin{center}
$\mathcal{X}[\mathbb{Z}]:=\cdots\cup \mathcal{X}[-1]\cup \mathcal{X}\cup \mathcal{X}[1]\cup \cdots$,
\end{center}
where $\mathcal{X}[n]:=\{X[n]\mid X\in \mathcal{X}\}$ for any $n\in \mathbb{Z}$. Since any element in $\mathcal{X}$ is a tilting complex as we mentioned in Example \ref{example-kronecker-alg}, so is each element in $\mathcal{X}[\mathbb{Z}]$. On the other hand, $\mathsf{Hom}_{\mathcal{K}_{\Delta_1}}(X_n,X_{n+1})\neq 0$ for any $n\in \mathbb{Z}$ (we set $X_{n+1}=X_0[1]$ if $n=-1$) by the definition of left mutations. Therefore, the silting complex of the form $(X_n\oplus X_{n+1}[k])[\ell]$ with $k\ge 1, \ell \in \mathbb{Z}$ is not tilting. It turns out that $\mathcal{X}[\mathbb{Z}]$ exhausts all tilting complexes in $\mathcal{K}_{\Delta_1}$, which is exactly the leftmost column in Figure 1.

By Lemma \ref{lem-symmetry} and Example \ref{example-kronecker-alg}, we find that $\mathsf{End}_{\mathcal{K}_{\Delta_1}}\ (X)\simeq \Delta_1$ or $\Delta_1^\text{op}$ for any $X\in \mathcal{X}$. Then, the statement follows obviously from the definition of $\mathcal{X}[\mathbb{Z}]$.
\end{proof}

\section*{Acknowledgements}
The author is grateful to Takuma Aihara for his insightful suggestions and kind help during the early stages of the author's study. The author would also like to thank Toshitaka Aoki and Kengo Miyamoto for many nice discussions.
The author is partially supported by the National Key R$\&$D Program of China (Grant No. 2020YFA0713000) and China Postdoctoral Science Foundation (Grant No. 315251).


 \ \\
\end{document}